\documentclass[preprint,12pt,authoryear]{elsarticle}

\usepackage{amssymb}
\usepackage{amsmath}
\usepackage{amsthm}
\usepackage{natbib}
\usepackage{xcolor,graphicx}
\usepackage{float}
\usepackage{caption}
\usepackage{subcaption}
\usepackage{epstopdf}

\numberwithin{equation}{section} 

\newtheorem{Example}{Example}[section]
\newtheorem{Lemma}[Example]{Lemma}
\newtheorem{Theorem}[Example]{Theorem}
\newtheorem{Definition}[Example]{Definition}
\newtheorem{Remark}[Example]{Remark}
\newtheorem{Proposition}[Example]{Proposition}
\newtheorem{Corollary}[Example]{Corollary}

\newcommand{\intd}{\,\mathrm{d}} 
\newcommand{\mids}{\;\middle\vert\;}
\renewcommand{\thefootnote}{\arabic{footnote}}

\newcommand{\nc}{\newcommand}
\nc{\rnc}{\renewcommand}
\nc{\nenv}{\newenvironment}
\nc{\mb}{\mathbb}
\rnc{\b}{\textbf}
\nc{\bs}{\boldsymbol}
\nc{\R}{\mb{R}}
\nc{\B}{\mathcal{B}}
\nc{\tn}{\textnormal}
\def\Rdp{\ensuremath [0,\infty)^d}
\def\Rdpo0{\ensuremath \Rdp \setminus \{\b 0\}}
\def\RIpo0{\ensuremath [0,\infty)^{|I|} \setminus \{\b 0\}}

\begin{document}

\begin{frontmatter}

\title{Exponent dependence measures of survival functions and correlated frailty models}

\author{J. Bendel$^{a}$, D. Dobler$^{b}$, A. Janssen$^{c,*}$\let\thefootnote\relax\footnote{${}^{*}$ Corresponding author; full address: Mathematical Institute, University of D\"usseldorf, Universit\"atsstr.~1, 40225 D\"usseldorf, Germany; Telephone number: 0049 211 8112165\\ 
  E-mail addresses: dennis.dobler@uni-duesseldorf.de (D. Dobler), jens.bendel@uni-duesseldorf.de (J. Bendel), janssena@uni-duesseldorf.de (A. Janssen) \\
}}

  \address{${}^{a}$Departement of Mathematics and Statistics, University of Reading, UK}
  \address{${}^{b}$Institute of Statistics, Ulm University, Germany}
  \address{${}^{c}$Mathematical Institute, University of D\"usseldorf, Germany}

\begin{abstract}

The present article studies survival analytic aspects of semiparametric copula dependence models
with arbitrary univariate marginals.
The underlying survival functions admit a representation via exponent measures 
which have an interpretation within the context of hazard functions.
In particular, correlated frailty survival models are linked to copulas.
Additionally, the relation to exponent measures of minumum-infinitely divisible distributions
as well as to the L\'evy measure of the L\'evy-Khintchine formula is pointed out.
The semiparametric character of the current analyses 
and the construction of survival times with dependencies of higher order are carried out in detail.
Many examples including graphics give multifarious illustrations.

\end{abstract}


\end{frontmatter}

\today

%
%
%
%
\section{Introduction} \label{sec:introduction}
Multivariate semiparametric dependence models have been successfully developed during the past decades.
We refer to the huge amount of copula literature; see \cite{Joe97} and \cite{Nelsen06} among others.
It is well-known that Sklar's (1959)\nocite{Sklar59} theorem allows 
to separate a multivariate distribution in a copula dependence part and marginal distribution functions
which are often considered as nuisance parameters.
On the other hand, failure rates and hazard functions are very meaningful in survival analysis for dependent life time data.
The famous Cox regression model specifies the dependence via exponential hazard dependence models 
with baseline hazard nuisance parameters, cf. \cite{Cox72}.
The hazard-based Cox models were extended by correlated frailty models 
which allow hazard dependence including flexible individual effects for life times;
see \cite{Duchateau08}, \cite{Hougaard00}, \cite{Wienke11} and \cite{Aalen08}.
The shared frailty model appears as a special case whose connection to Archimedean copulas was pointed out by
\cite{Kimberling74}, \cite{Marshall88}, \cite{Oakes89} and \cite{McNeil09};
see also \cite{Genest86}, \cite{Duchateau08} and \cite{Voelker10} for further discussions.

It is the aim of this paper to study the multivariate dependence structure from the survival analysis point of view,
in particular for the survival functions of copulas.
We thereby like to connect the copula dependence structure with hazard quantities.
Our hazard dependence functions (and dependence measures) have a direct interpretation in survival analysis
and they may serve as parameters for statistical dependence modeling.
Recall that, under continuity, univariate hazard measures $\Lambda$ are exponent measures of survival functions $S$, 
i.e. $S(t) = \exp(-\Lambda(t))$ for continuous $S$; 
see Gill (1993)\nocite{Gill93a, Gill93b} and \cite{Dabrowska96} for exponents of multivariate survival functions.

The article is organized as follows. 
The basics of multivariate survival analysis are presented in Section~\ref{sec:surv_func}.
Along with the work of Gill (1993)\nocite{Gill93a, Gill93b} and \cite{Dabrowska96}
the exponent representation of Lemma~\ref{Lemma1-2} for multivariate survival functions is recalled
which later leads to so-called exponent measures of dependence.
It is also connected to the concept of ``local dependence hazard'' parameters studied earlier by \cite{JR02} in the bivariate case.
Section~\ref{sec:exponent_measures_of_dependence} links the exponent dependence measure 
to earlier analyzed exponent measures of \cite{Resnick87} for minumum-infinitely divisible distributions in extreme value theory.

In Section~\ref{sec:frailty} signed exponent measures of dependence are studied in detail 
for survival functions of correlated frailty models.
These are introduced as multivariate exponentially distributed variables with random scales $\b W=(W_1,\dots,W_d)$.
In terms of the Laplace transform of $\b W$ explicit analytic formulas are derived for their densities, hazards and exponent measures,
in particular, for the special examples of dependent $\chi^2$ or log-normal random scales $W_i$.
If $\b W$ is sum-infinitely divisible, then the dependence exponent measures are linked to the L\'evy measure
given by the L\'evy-Khintchine formula for $\b W$; see Section~\ref{sec:frailty_id}.

Section~\ref{sec:semiparametric_dependence_models} is devoted to semiparametric survival models with arbitrary marginals.
It is shown that the exponent dependence measures can be parametrized in a copula manner.
The underlying parametrized dependence quantities are later visualized 
for some examples of correlated frailty models; see Figures~\ref{fig:gamma_0_fuer_clayton} to~\ref{fig:gamma_0_fuer_Zj2}.
Together with the motivation given by the upcoming equations~\eqref{1Punkt12}--\eqref{eq:interpretation_lambda_12_local_dep_hazard}
the plots have a meaningful hazard-based interpretation regarding the quality of dependence.
For instance, there exist copulas with overall proportional hazard dependence for all bivariate submodels
but without dependence for exponent measures of higher order.
On the contrary, copula models can be constructed 
for which all proper subvectors are independent but the whole vector has a non-trivial dependence measure;
see Section~\ref{sec:higher_dep} for an elaboration of such models via hazard functions.
Additional examples and graphs as well as some technical details are presented in the appendix to this article.

%
%
%
%
\section{Survival Functions} \label{sec:surv_func}
A copula $C$ is a distribution function on the $d$-dimensional unit cube $[0,1]^d$ with uniformly distributed marginals. 
It is well known that copulas describe the dependence structure of $ \mathbb{R}^d$-valued random variables
$\b T=(T_1, \ldots , T_d)$ whereas the distributions of the one-dimensional marginals $T_i $ are often regarded as 
nuisance parameters in statistics.

In this paper we will study survival functions and their corresponding survival copulas which are very useful in risk and survival analysis 
when multivariate failure rates and hazard rates have a concrete meaning; see for instance Gill (1993)~\nocite{Gill93a, Gill93b} and \cite{Dabrowska96}. 
For a pair of random variables the hazard rate dependence  approach is a meaningful description of the 
survival copula dependence structure; see \citet{JR02}.
They showed that an exponential dependence measure exists which naturally explains the 
dependence structure.
At other times we like to apply the copula concept in order to convert all marginals to uniform distributions,
eliminating all nuisance parameters in doing so.

Throughout let $\b T=(T_1, \ldots , T_d):\Omega \rightarrow \mathbb{R} ^d$ denote a random variable 
with $d$-dimensional continuous distribution function $F$
and let $F_i$ and $S_i(t_i) = 1-F_i(t_i) = P(T_i>t_i)$, $t_i \in \mathbb{R}$,
be the marginal distribution and survival functions.
The $d$-dimensional survival function is given by
\begin{equation}\label{1p1}
	S(t_1,\ldots,t_d)=\text{P}(T_1>t_1,\ldots,T_d>t_d).
\end{equation}
To fix the notation let $ F_i^{-1} $ ($S_i^{-1}$) denote the left-sided continuous inverse distribution (survival) 
function of $F_i$ $(S_i)$. In our continuous case the variables
\begin{equation}\label{eq:trans_marg}
	\text{(a)}\ (F_1(T_1),\ldots,F_d(T_d))  \quad \text{and} \quad \text{(b)}\ (S_1(T_1),\ldots,S_d(T_d))
\end{equation}
both define copulas on $[0,1]^d$:
The distribution function $C$ of (a) is simply called \emph{copula} of $\b T$ and similarly 
the distribution function $C_s$ of (b) is defined as the \emph{survival copula} of $\b T$, cf. \cite{Nelsen06}, Section 2.6, as well as \cite{McNeil09}.
We see that
\begin{equation*}
	(v_1,\ldots,v_d) \mapsto C_s(1-v_1,\ldots,1-v_d)
\end{equation*}
is just the survival function (\ref{1p1}) of $C$. 
To avoid further confusions we like to stress the necessity to clearly distinguish 
between the survival copula $C_s$ and the survival function of the copula $C$.
These concepts lead to different dependence models for survival times.

Furthermore, the survival function $S$ and the survival copula are connected by
\begin{equation}\label{1p5}
	S=C_s(S_1,\ldots,S_d)\quad \text{and} \quad C_s=S\left(S_1^{-1},\ldots,S_d^{-1}\right).
\end{equation}
Thus, every survival function $S$ can easily be reproduced by its survival copula $C_s$ and the continuous 
marginal survival functions. 

In order to study multivariate survival functions in more detail we first recall elements 
of univariate survival analysis for a continuous real valued random variable $T$. 
For a detailed introduction we refer to the books by \cite{Klein03} and \cite{Aalen08}.
If we put $x_0=\text{inf}\{x:\text{P}(T\le x)>0  \}$, then there exists a so-called hazard measure $\Lambda$ 
on the interval $(x_0,\infty)$ with cumulative hazard function $\Lambda (t):=\Lambda (-\infty,t]$ given by
\begin{equation} \label{eq:one_and_anoter_one}
	\frac{\intd\Lambda}{\intd F} = \frac 1 S \quad  \text{and} \quad  P(T>t)= 1 - F(t) = S(t) = \exp(- \Lambda (t)).
\end{equation}
The hazard measure $\Lambda$ can be viewed as univariate exponent measure of the survival function.
Whenever $F$ has a Lebesgue density $f$ then $\Lambda(t)=\int_{x_0}^{t} \lambda (u) \ \mathrm{d}u $ 
holds for the hazard rate $\lambda(u):=\frac{f(u)}{S(u)}$. 
For example, the hazard measure of a standard exponentially distributed random variable is just 
the Lebesgue measure $\lambda \!\! \lambda|_{(0,\infty)}$ on $(0,\infty)$. 

In the next step exponential representations of multivariate survival functions are studied which are used to extend the univariate hazard formula~\eqref{eq:one_and_anoter_one}. 
Special models and exponent measures are examined in the proceeding sections. 
Consider an index set $\emptyset \neq I\subset I_0 := \{1,\ldots,d\}$. Then 
\begin{equation*}
  \Omega_I:=\lbrace(t_i)_{i \in I}\in \mathbb{R}^{|I|} : \text{P}(T_i>t_i \ \text{for all i} \in I) >0 \rbrace
\end{equation*}
denotes the domain of the survival function of the marginals $(T_i)_{i \in I}$. 
The following Lemma 
is written in the spirit of Gill (1993)~\nocite{Gill93a,Gill93b} who studied the survival function in the context of product integrals. 
The proof follows by induction;
see also \cite{Dabrowska96}, equation~(1.1).
\begin{Lemma} \label{Lemma1-2}
	Consider the continuous survival function $S$ of $\b T=(T_i)_{i \in I_0}$ 
	and let $S^J$ denote the survival function of the $J$-dimensional marginal vector $(T_j)_{j \in J}, J \subset I_0$.
	Define for a single subset 
	$I=\lbrace i \rbrace$ the function $S_{I}$ on $\Omega_{I}$ as the marginal 
	$S_{I}=S_i$. For every $I \subset I_0$, $|I| \geq 2$, 
	there exists a positive function $S_{I}:\Omega_{I}\rightarrow \mathbb{R}$ 
	which is uniquely determined by the following representation of marginal survival functions.
	For each index set $\emptyset \neq J \subset I_0$ the marginal survival function
	admits the factorization
	\begin{equation}\label{1p9}
		S^J=\prod_{I \subset J} S_{I} \quad \text{on } \Omega_{J}
	\end{equation}
	where the product is taken over all subsets 
	$\emptyset \not= I \subset J$. 
	We call 
	$(S_{I})_{I, |I|=r}$ the dependence parts of order $2 \le r \le d$ which 
	are unique on their domain.
\end{Lemma}

For a singleton $I=\lbrace i \rbrace$ the cumulative hazard function $\Lambda _i$ corresponds to an exponent measure, 
$S_i=\exp(-\Lambda_i)$. In the same manner we can, for all ${I} \subset I_0$ with $|I| \geq 2$, define an exponent function $\Lambda_{I}$ on the domain 
$\Omega_{I}$ by
\begin{equation} \label{eq:quasi_Def_von_Lambda_I}
	S_{I}=\exp \left( (-1)^{|I|}\Lambda _{I} \right).
\end{equation}
As we will see, the higher dimensional  exponents are typically linked to signed exponent measures of dependence.
As motivation and in order to give a proper interpretation let us first recall 
the hazard dependence approach for $T_i \geq 0$ and dimension $d=2$:
\cite{JR02}
pointed out that $\Lambda_{\lbrace 1,2 \rbrace}$ is a signed exponent measure. 
If $S$ has a density then 
$\Lambda_{\lbrace 1,2 \rbrace}$ also has a density $\lambda_{\lbrace 1,2 \rbrace}$ on 
$\mathbb{R}^2$ and (\ref{1p9}) writes as
\begin{equation}\label{1Punkt12}
	S(t_1,t_2) =  S_1(t_1)  S_2(t_2)  \exp\left(\int_0^{t_1} \int_0^{t_2} \lambda_{\lbrace1,2\rbrace}(u_1,u_2)  \intd u_1 \intd u_2 \right).
\end{equation} 
Recall again
the interpretation of the univariate hazard 
function $\lambda_1(u)=\frac{f_1(u)}{S_1(u)}$ 
as failure rate. 
Under smoothness of $f_1$ and for small $\epsilon > 0$ we approximately have,
given the event $\{T_1 \geq t_1\}$,
\begin{equation}
	\text{P}(T_1\in [u,u+\epsilon] \ | \ T_1\geq u)  =  \frac{\int_u^{u+\epsilon}f_1(x) \intd x}{S_1(u)}   \approx  \epsilon  \lambda_1(u).
\end{equation}
In the same way $\lambda_{\lbrace1,2\rbrace}$ has an interpretation as ``local dependence hazard" of 
$T_1$ and $T_2$. 
For smooth densities, $B=\lbrace T_1 \geq t_1, T_2 \geq t_2 \rbrace$ and small $\epsilon_1,\epsilon_2 >0$ \cite{JR02} verified that
\begin{equation} \label{eq:interpretation_lambda_12_local_dep_hazard}
	\begin{split} 
		&\text{P} (T_1\in [t_1,t_1+\epsilon_1], T_2\in [t_2,t_2+\epsilon_2]\ | \ B) \\
		& -\text{P} (T_1\in [t_1,t_1+\epsilon_1] \ | \ B) \ \text{P}(T_2\in [t_2,t_2+\epsilon_2] \ | \ B) \\
		&\quad \approx \ \lambda_{ \lbrace 1,2 \rbrace} (t_1,t_2)\  \epsilon_1 \ \epsilon_2
	\end{split}
\end{equation}
holds.
They also showed that 
the exponent dependence measure and its density are extremely useful for testing independence for 
randomly censored survival models. 
In their work the reader will find a lot of examples for hazard dependence structures
and efficient tests for independence of the marginals at dimension 2; see also Section~\ref{sec:higher_dep} below.
%
%
%
\section{Exponent dependence measures of minimum-infinitely divisible distributions} \label{sec:exponent_measures_of_dependence}
In this section we link the exponents of dependence $\Lambda_I$ to already known measures of dependence. 
Observe that in extreme value theory mini\-mum-infinitely divisible distributions allow related exponent representations.

A distribution on $\mathbb R^d$ with survival function  $S$ is called minimum-infinitely divisible if $S^{1/n}$ is a survival function for each $n\in\mathbb N$.
For obvious reasons we will focus on the copula case.
It is well known \citep[cf.][Section~5.3, switching from the $\max$ to the $\min$ operation]{Resnick87} 
that, under minimum-infinite divisibility, there exists a possibly unbounded exponent Radon measure $\mu$ on $[0,1]^d\setminus \{\b 1 \}$ with 
\begin{align} \label{eq:S=exp(-mu)_min_id}
	S(\b x) = \exp\left( -\mu\left( (\b x, \b 1]^{\mathcal C} \right) \right),
\end{align}
where $(\b x, \b 1]^{\mathcal C}$ is the complement of $(x_1,1] \times \dots \times (x_d,1]$ in $[0,1]^d$
and $\b 1 = (1,\dots,1) \in \R^d$.
To get an impression of \eqref{eq:S=exp(-mu)_min_id} consider the bivariate case $d=2$:
\begin{Example} \label{example:biv_mu_example}
	(a) 	Let $S = S_1 S_2$ be the survival function of the independence copula of dimension $2$.
		Then $\mu = \Lambda_1 \otimes \varepsilon_1 +\varepsilon_1 \otimes \Lambda_2$ lies on the upper 
		boundary of $[0,1]^2$ with $\Lambda_i(t) = \int_0^t \frac{\intd x}{1-x}$, $t \in [0,1)$,
		and $\varepsilon_1$ being the Dirac measure in 1.
	\\
	(b) 	If $S = S_1 S_2 \exp\left( \Lambda_{\{1,2\}} \right) $ is minimum-infinitely divisible and given by \eqref{eq:S=exp(-mu)_min_id} then
		\begin{align*} 
			\Lambda_{\{1,2\}} \left( [0,x_1] \times [0,x_2] \right) = \mu\left( [0,x_1] \times [0,x_2] \right)
		\end{align*}
		holds for all $0\leq x_i < 1$ if $S$ is continuous.
\end{Example}
\begin{Remark}
	(a) 	The distribution of two non-negative random variables with local dependence hazard
		$\lambda_{\{1,2\}}$ 
		(cf. Proposition~\ref{prop:bivariate_lambda} and the discussion after 
		Lemma~\ref{Lemma1-2})
		is minimum-infinitely divisible iff
		$\lambda_{\{1,2\}}$ is non-negative. \cite{Joe97}, Theorem~2.7, proves this and a 
		similar result for arbitrary dimensions $d \geq 2$.
	\\
	(b) 	Part (b) of Example~\ref{example:biv_mu_example} is a special case of the following Proposition~\ref{prop:identification_Lambda-I-s_vs_mu-I-s},
		from which we conclude that
		the behaviour of $\mu$ in the interior of $[0,1]^d$ determines the dependence structure of $S$.
\end{Remark}
\begin{Proposition} \label{prop:identification_Lambda-I-s_vs_mu-I-s}
	Let $S$ be the survival function of a continuous minimum-infinitely divisible copula given by \eqref{eq:S=exp(-mu)_min_id}.
	For each index set $\emptyset \neq I \subset I_0$ the function $\Lambda_I$ given by 
	\eqref{eq:quasi_Def_von_Lambda_I} is a measure generating function of a positive measure on $[0,1)^{|I|}$ with 
	\begin{align} \label{eq:identification_Lambda-I-s_vs_mu-I-s}
		\Lambda_I = \mu_I\vert_{[0,1)^{|I|}}
	\end{align}
	where $\mu_I := \mathfrak{L}\left( \pi_I \mids \mu\vert_{ [0,1]^d \setminus \pi_I^{-1} (\{(1,\dots,1)\}) }\right)$ 
	is the image measure of the canonical projection $\pi_I : [0,1]^d \to [0,1]^{|I|}$ w.r.t. the restricted measure of $\mu$
	on $[0,1]^d \setminus \pi_I^{-1} (\{(1,\dots,1)\})$.
\end{Proposition}
\begin{proof}
	The proof relies on an inclusion-exclusion principle type of argument (also known as the sieve formula or Sylvester-Poincar\'e  equality). 
	For $\b x\in [0,1)^d$ observe that
	\begin{align*}
		\b 1_{(\b x, \b 1]^{\mathcal C}} = \sum_{\emptyset \neq J \subset I_0} (-1)^{|J|+1} \prod_{j\in J} \b 1_{[0,x_j]}
	\end{align*}
	and
	\begin{align*}
		\mu \left( (\b x,\b 1]^{\mathcal C} \right) = \sum_{\emptyset \neq J \subset I_0} (-1)^{|J|+1} \mu_J \Big( \prod_{j\in J} [0,x_j] \Big)
	\end{align*}
	hold, where $\b 1_A$ denotes the indicator function of a set $A$. 
	Moreover, $\mu_J$ is the exponent measure \eqref{eq:S=exp(-mu)_min_id} of the $|J|$-dimensional marginal survival function $S^J$. 
	Hence, $\Lambda_I$ and $\mu_I\vert_{[0,1)^{|I|}}$ can be identified step by step. For $|I| = 1 $ the univariate hazard measures are obtained. 
	In genereal, when all $\Lambda_J$ with $J\subset I$, $J\neq I$, are identified, 
	then the remaining part $\Lambda_I$ must be given by \eqref{eq:identification_Lambda-I-s_vs_mu-I-s}.
	Note that $S$ is completely determined by $S(\b x)$ for $\b x\in [0,1)^d$.
\end{proof}
\begin{Remark}
		Suppose that $\b T = (T_1, \dots, T_d)$ has a minimum-infinitely divisible distribution with exponent measure $\mu$ given by \eqref{eq:S=exp(-mu)_min_id}
		and that $S$ is continuous. 
		The following statements are equivalent (cf. \cite{Resnick87}, Section~5.5, for the equivalence of $(i)$ and $(ii)$).
		\begin{enumerate}[(i)]
			\item The components $T_1,\dots,T_d$ are independent random variables.
			\item The components $T_1,\dots,T_d$ are pairwise independent, i.e. $T_i$ and $T_j$ are independent random variables for every $1\leq i < j \leq d$.
			\item For every $I\subset I_0$ with $|I|=2$ the bivariate exponent measure $\Lambda_I$ given by \eqref{eq:quasi_Def_von_Lambda_I} is equal to zero.
			\item For every $I\subset I_0$ with $|I|\geq 2$ the exponent measure $\Lambda_I$ given by \eqref{eq:quasi_Def_von_Lambda_I} is equal to zero.
		\end{enumerate}
\end{Remark}
%
%
%
%
%
\section{Dependence measures for correlated frailty models}
\label{sec:frailty}
As explained in Section~\ref{sec:surv_func} the univariate marginals are irrelevant when the dependence structure is studied 
within a semiparametric context. This is a chance to choose in a first step the marginals in a convenient
way as we did in~\eqref{eq:trans_marg}. In a second step the marginals may be transformed if necessary. For this reason 
we will first introduce a correlated frailty model for special marginals which gives more insight in the 
underlying dependence structure. Originally the frailty concept shows up in multivariate survival 
analysis when individual hazard effects are modeled by an additional factor attached to famous Cox 
models; see \cite{Andersen93} for further references.
The idea to link several individuals via the same factor leads to shared frailty models
and closely related Archimedean copulas; see Section~\ref{sec:introduction} for references to this connection.
Our next example shows that these particular copulas fit very well in the context of survival copulas.
\begin{Example}  \label{example:archimedean_copulas}
	 Let $\varphi: [0,\infty) \rightarrow [0,1]$ be the following generator of an Archimedean copula $K$: 
	$\varphi$ is a nonincreasing and continuous function satisfying
	$\varphi (0)=1$,\ $\lim_{x \rightarrow \infty}{\varphi(x)}=0$ and strictly decreasing 
	on the interval $[0, \inf\{x:\varphi(x)=0 \})$. Then $K$ is given by 
	\begin{equation*}
		K(u_1,\dots,u_d)=\varphi\left(\varphi^{-1}(u_1)+\dots+\varphi^{-1}(u_d)\right).
	\end{equation*}
	Obviously, $S_i := \varphi $ define marginal survival functions. If $(\zeta_1,\dots,\zeta_d)$ is 
	distributed according to the above Archimedean copula $K$ with generator $\varphi$ then by (\ref{1p5}) 
	\begin{equation}\label{1p8}
		(t_1,\dots,t_d) \mapsto \varphi(t_1+\dots+t_d)
	\end{equation}
	is the survival function of $\b T=(\varphi^{-1}(\zeta_1),\ldots,\varphi^{-1}(\zeta_d))$ with survival function $S_i=\varphi$ 
	of the marginals.
	Furthermore, $C_s = K$ is the survival copula of $\b T$.
\end{Example}
We see that the above choice of marginals yields a pleasant form of the survival function.
\begin{Remark}
  \citet{Okhrin13} study hierarchical Archimedean copulas, 
  which extend the class of Archimedean copulas. 
  They show that any hierarchical Archimedean copula can be uniquely recovered from all bivariate marginal copula functions.
\end{Remark}

We now introduce a class of correlated frailty models which admit meaningful hazard dependence structures.
The explanation below Lemma~\ref{lemma:surv_eq_lapl} exhibits the relation to the Archimedean copula structure~\eqref{1p8}.
Let $Q_\b 0$ denote a probability measure on $(0,\infty)^d$ with Laplace transform 
\begin{equation*}
	\psi(\b t):=\int \exp(-<\b t,\b x>) \ \intd Q_\b 0(\b x), \ \ \ \b t=(t_1,\ldots,t_d) \in [0,\infty)^d.
\end{equation*}
It will turn out that $\psi$ is a meaningful quantity for the multivariate survival time $\b X$ defined as follows.
\begin{Definition}
	Let $\b W=(W_1,\ldots,W_d)$ and $(Y_1,\ldots,Y_d)$ denote two independent $d$-dimensional random vectors with law
	$\mathfrak{L}(\b W)=Q_\b 0$ and let $Y_1,\ldots,Y_d$ be i.i.d. standard exponentially distributed. The random vector $\b X$ given by
	\begin{equation}\label{eq:2p2}
		\b X=(X_1,\ldots,X_d) := \left(\frac{Y_1}{W_1},\ldots,\frac{Y_d}{W_d}\right)
	\end{equation}
	is then called a \emph{correlated frailty model} based on $Q_\b 0$.
\end{Definition}
\begin{Remark}
	In multivariate survival analysis the concept of frailty was originally introduced 
	with equal variables $W_1 = W_2 = \dots = W_d$ (called \emph{shared frailty model}) 
	to model unobserved heterogeneity.
	For a more detailed discussion of the univariate \emph{proportional frailty model}, 
	including surveys of the model for different distributions of the frailty variable and extensions, 
	we refer to the books by 
	\cite{Wienke11}, Chapter~3,
	\cite{Duchateau08} and 
	\cite{Aalen08}, Chapter~6,
	as well as to the article by \cite{Voelker10}.
	In contrast to the shared frailty model, the \emph{correlated frailty model} is more flexible
	and it takes the dependency structure of the $W_i$'s into account.
	
	Usually the definitions of frailty models are given in terms of the hazard rates $\lambda_1,\dots,\lambda_d$ of $X_1,\dots,X_d$.
	Thus it is possible to include a baseline hazard as well as covariates in the model.
	As the focus of this article lies on dependence structures we omit covariates and choose (justified by the copula approach) 
	a possibly simple baseline hazard equal to one. These simplifications lead to the definition given by \eqref{eq:2p2}.
	
	A brief introduction to the correlated frailty model for the bivariate case (d = 2) 
	and its relation to Cox models can be found in the book by \cite{Wienke11}, Section~5.1.
\end{Remark}
The correlated frailty model is a multivariate exponential scale model with random scale parameter $\b W$. 
Any strictly increasing transformation of the coordinates of $\b X$ leads to a conditional Cox model given $\b W$. 
Subsequently, it turns out that the structure of \eqref{eq:2p2} is closely connected to the exponential family $(Q_\b t)_{\b t \in [0,\infty)^d}$ given by 
\begin{equation}\label{Formel2-3}
	\frac{\mathrm{d}Q_\b t}{\mathrm{d}Q_\b 0} (\b x) := \frac{\exp(-<\b t,\b x>)}{\psi(\b t)}.
\end{equation}
The following analytic properties of Laplace transforms and exponential families are well known; see  \cite{Barndorff78}.
\begin{Remark}
	(a) 	The Laplace transform $\b t \mapsto \psi(\b t)$ is a positive analytic function on $(0,\infty)^d$.
	\\
	(b) 	Consider $I \subset I_0$  and, for each $i \in I$, let $n_i \in \mathbb{N}_0$ be a multiplicity. 
		Let $\b t=(t_1,\ldots,t_d) \in [0,\infty)^d$ with $t_i>0$ whenever $n_i > 0$ holds. 
	Then
	\begin{equation}\label{2p4}
		\bigg( \prod_{i \in I} \frac{\partial^{n_i}}{\partial t_i^{n_i}} \bigg) \psi(\b t)= \psi(\b t) \int \prod_{i \in I} (-1)^{n_i} \ x_i^{n_i} \ \mathrm{d}Q_\b t(\b x).
	\end{equation}
	\\
	(c) 	Suppose that $(Y_1, \dots, Y_d): \Omega' \rightarrow \R^d$ is given by a probability measure $P'$ on $\Omega'$.
		Via projection we may assume that the underlying joint probability space is a product space 
		$\Omega=\Omega' \otimes (0,\infty)^d$ with product measure $P_\b 0 =P'\otimes Q_\b 0$ and $\b W(\omega',\b x)=\b x$ 
		is the identity of the second component. Then $(W_1,\ldots,W_d)$ can naturally be regarded as 
		random vector under $P_\b t=P' \otimes Q_\b t$, shortly under $Q_\b t$.
\end{Remark}
\begin{Lemma} \label{lemma:surv_eq_lapl}
	The following results hold for the correlated frailty model $\b X$.
	\\
	(a) The survival function $S(\b t)=\psi(\b t)$ coincides with the Laplace transform for all 
	$ \b t \in [0,\infty)^d$. 
	\\
	(b) The distribution of $\b X$ is concentrated on $(0,\infty)^d$. On this set it has the analytic density $f$
	\begin{equation}\label{2p5}
		\b t \mapsto f(t_1,\ldots,t_d) = (-1)^d \ \frac{\partial}{\partial t_1} \ldots \frac{\partial}{\partial t_d} \  \psi(t_1,\ldots,t_d).
	\end{equation}
	(c) For $t_i > 0$ and $\b t=(0,\ldots,0,t_i,0,\ldots,0)$ the $i$-th marginal hazard measure $\Lambda_i$ 
	  is given by $\Lambda_i(t_i)=-\log\psi(\b t)$ with $i$-th univariate hazard rate
	\begin{equation}\label{Formel2-6}
		\lambda_i (t_i) = - \frac{\partial}{\partial t_i} \log \psi (\b t) = \int W_i  \intd Q_\b t .
	\end{equation}
\end{Lemma}
Consider the special case that $W_1 = \dots = W_d$ coincide, i.e. $\b X$ is a shared frailty model. 
If $\varphi$ denotes the univariate Laplace transform of $W_1$,
then the survival function of $\b X$ collapses to $S(\b t) = \psi(\b t) = \varphi(t_1 + \dots + t_d)$ 
and~\eqref{eq:2p2} corresponds to the transformation of the Archimedean copula~\eqref{1p8}
given by the completely monotone function $\varphi$.
\begin{proof}[Proof of Lemma~\ref{lemma:surv_eq_lapl}]
	(a) Using Fubini's theorem we obtain
	\begin{eqnarray*}
		S(\b t) 	&=& \int \ \text{P}'\left(\frac{Y_1}{w_1}>t_1, \ldots, \frac{Y_d}{w_d}>t_d\right) \ \mathrm{d}Q_\b 0(w_1,\ldots,w_d) \\
			&=& \int \exp (-t_1w_1-\ldots-t_dw_d) \ \mathrm{d}Q_\b 0(w_1,\ldots, w_d)
	\end{eqnarray*}
	by the independence of the $Y_i$'s. \\
	(b) Whenever a density $f$ exists, we have
	\begin{equation*}
		\psi (\b t)=S(\b t)=\int_{t_1}^{\infty} \ldots \int_{t_d}^{\infty} f(u_1,\ldots,u_d) \ \mathrm{d}u_1\ldots \mathrm{d}u_d
	\end{equation*}
	by part (a) for each $\b t \in (0,\infty)^d$.
	Formal differentiation of $\b t\mapsto \psi(\b t)$ gives the analytic left-hand side of (\ref{2p5})
	and the result follows. \\
	(c) The first part follows since 
	$
		\Lambda _i (t_i) = -\log S_i(t_i) = - \log \psi(\b t).
	$
	Then its derivative $\frac{\partial}{\partial t_i} \Lambda_i(t_i)=\lambda_i(t_i)$ can be calculated by (\ref{2p4}).
\end{proof}
We see that distributional quantities like $S$ and $f$ are linked to $\psi$ and its derivatives. 
To see that the hazard quantities also rely on $\log \psi$, consider $S$ as in (\ref{1p9}). 
Then we have
\begin{equation*}
	\log S(\b t)=\sum_{I \subset I_0} \log S_{I} (\b t)=- \sum _{i=1}^d \Lambda _i(t_i) +\sum_{ \substack{ I \subset I_0 \\ |I|>1 }} (-1)^{|I|} \Lambda_I(\b t_I)
\end{equation*}
with first order terms $\Lambda_i(t_i)$.

The higher order terms for $|I|>1$ are now studied in detail and linked to (\ref{1Punkt12}). 
Throughout, the following notation is used.
Let $h:[0,\infty)^d \rightarrow \mathbb{R}$ denote a function with existing derivatives of sufficiently high order. 
Introduce for $\emptyset \neq I \subset I_0$ the following notation: For $\b x_I \in \mathbb{R}^{|I|}$
\begin{equation*}
\b x \mapsto h(\b x_I,\b 0)
\end{equation*}
is the function given by $h$ when the coordinates $x_j$, $j \notin I$, are zero.
Similarly, for a given $\b x \in \R^d$ we denote by $(\b x_I, \b 0)$ the $d$-dimensional vector with zero entries for all indices $j \notin I$
and with entries $x_i$ else.
For $I=\{ i_1, \ldots, i_r\}$ we also introduce
\begin{equation*}
	\frac{\partial}{\partial \b x_I} h(\b x_I,\b 0) = \frac{\partial}{\partial x_{i_1}} \ldots \frac{\partial}{\partial x_{i_r}} h(\b x_I,\b 0)
\end{equation*}
and
\begin{equation*}
	\int_{\b 0}^{\b y_I}h(\b x_I,\b 0) d\b x_I  = \int_{0}^{y_{i_1}} \ldots \int_{0}^{y_{i_r}} h(\b x_I,\b 0) \intd x_{i_1} \ldots \intd x_{i_r} , \quad \b y_I \in \R^{|I|}.
\end{equation*}
\begin{Lemma}
The logarithm of the survival function $S$ of a correlated frailty model is given by
\begin{equation}\label{FormelLemma24}
\log{S(\b t)}=\sum_{\emptyset \neq I \subset I_0} \int_{\b 0}^{\b t_I} \frac{\partial}{\partial \b x_{I}} \log{\psi(\b x_I,\b 0)} \intd \b x_I.
\end{equation}
\end{Lemma}

\begin{proof}
Since $\log{\psi(\b 0)=0}$ holds, the main theorem of multivariate calculus yields (\ref{FormelLemma24}) by induction over $d$.
\end{proof}

This result identifies the exponents $\Lambda_I$ for $I\neq \emptyset$. We have
\begin{equation}\label{Formel2-11}
	\Lambda_I(\b t_I)= \int_{\b 0}^{\b t_{I}} \lambda_I(\b x_I) \intd \b x_I
\end{equation}	
where $\lambda_I(\b x_I) = (-1)^{|I|}  \frac{\partial}{\partial \b x_{I}} \log \psi(\b x_I,\b 0).$
Similarly as for Lemma \ref{Lemma1-2} the proof is given by induction over the dimension $d$. 
The first order (hazard) measures are described in (\ref{Formel2-6}) by the exponential family (\ref{Formel2-3}). 
In our context $\Lambda_I$ are measure generating functions of signed measures.
Here, the bivariate measures $\Lambda_I, |I| = 2,$ are determined by particularly interesting densities:
\begin{Proposition} \label{prop:bivariate_lambda}
	Consider $I=\{ i,j\} \subset I_0 , i\neq j$, and $\b t_I \in (0,\infty)^2$.
	Then the density $\lambda_I$ of the signed measure $\Lambda_I$ is given by the covariance structure of the exponential family:
	\begin{equation*}
		\lambda_I (\b t_I) = Cov_{Q_{(\b t_I,\b 0)}}(W_i,W_j).
	\end{equation*}
\end{Proposition}

\begin{Remark}
	(a) 	The local hazard dependence density $\lambda_{\{1,2\}}$ in~\eqref{eq:interpretation_lambda_12_local_dep_hazard} 
		has now an interpretation in terms of correlations for the correlated frailty model. 
	\\
	(b) 	Higher order expressions of \eqref{Formel2-11} are similar but more complicated.
		The role of higher order exponent measures is discussed in Section~\ref{sec:semiparametric_dependence_models}.
\end{Remark}
\begin{proof}[Proof of Proposition~\ref{prop:bivariate_lambda}]
	For simplicity we may consider $I=\{1,2\}$. By (\ref{Formel2-11}) we have
	\begin{equation*}
		\lambda_I (\b t_I) =
		\frac{\partial}{\partial \b t_{I}} \log{\psi(\b t_I,\b 0)} = 
			\frac{\frac{\partial}{\partial t_{1}} \frac{\partial}{\partial t_{2}} \psi(\b t_I,\b 0) }{\psi(\b t_I,\b 0)} 
			- \biggl( \frac{\frac{\partial}{\partial t_{1}} \psi(\b t_I,\b 0) }{\psi(\b t_I,\b 0)} \biggr) 
				\biggl(\frac{\frac{\partial}{\partial t_{2}} \psi(\b t_I,\b 0) }{\psi(\b t_I,\b 0)} \biggr) .
	\end{equation*}
	Combining (\ref{Formel2-3}) and (\ref{2p4}), 
	we see that the first term corresponds to the expected value $E_{Q_{(\b t_I, \b 0)}}(W_1W_2)$ 
	whereas the second term is just the product $E_{Q_{(\b t_I, \b 0)}}(W_1)E_{Q_{(\b t_I,\b 0)}}(W_2)$;
	see also (\ref{Formel2-6}).
\end{proof}

In life science analysis two risk components can be modeled as a minimum of two survival variables:
For example, we could think of individuals who are exposed to two different lethal risk factors.
Then the survival times are the first occurrences of one of the competing events.
This concept goes well with the correlated frailty model:
\begin{Lemma} \label{lem:min_frailty}
	Let $\b W$ and $\b W'$ be two $d$-dimensional frailty variables 
	with correlated frailty models
	\begin{align*}
		\b X_{\b W} = \left( \frac{Y_1}{W_1}, \dots, \frac{Y_d}{W_d} \right) 
		\text{ and }
		\b X_{\b W'} = \left( \frac{Y'_1}{W'_1}, \dots, \frac{Y'_d}{W'_d} \right)
	\end{align*}
	for i.i.d. standard exponentially distributed $Y_1,\dots,Y_d,Y'_1,\dots Y'_d$.
	Let $S_{\b W}$ and $S_{\b W'}$ denote the survival functions of $\b X_{\b W}$ and $\b X_{\b W'}$, respectively.
	\\
	(a) We have equality in distribution of 
	  \begin{align*}
	    \b X_{\b W + \b W'}:= \left( \frac{Y_1}{W_1 + W'_1}, \dots, \frac{Y_d}{W_d + W'_d} \right) 
	      \stackrel{\mathfrak{D}}{=} \min({\b X}_{\b W}, {\b X}_{\b W'})	  
	  \end{align*}
	  Here $\min$ denotes the component-wise minimum operation.
	  Thus, the class of correlated frailty models is closed w.r.t. the minimum operation.
	\\
	(b) The following conditions (i)-(ii) are equivalent.
	\begin{itemize}
		\item[(i)]	The correlated frailty model 
				$\b X_{\b W + \b W'}$ has the survival function $S_{\b W} S_{\b W'}$.
		\item[(ii)]	$\b W$ and $\b W'$ are independent.
	\end{itemize}
\end{Lemma}
The proof is obvious. Note, that the conditional survival function of $\min({\b X}_{\b W}, {\b X}_{\b W'})$
at $\b t$ given $(\b W, \b W') = (\b w, \b w')$ is $\exp(- \! < \! \b t, \b w + \b w' \! > \!)$.
The independence of $\b W$ and $\b W'$ is equivalent to the product of 
Laplace transforms which corresponds to the products of survival functions; see Lemma~\ref{lemma:surv_eq_lapl}.
\begin{Example}\label{Bsp:doppelt}
	Let $(Z_1,\dots,Z_d)$ have a multivariate normal distribution 
	with mean vector zero and non-singular covariance matrix $\bs \Sigma$.
	\\
	(a) 	Correlated frailty models with the following frailty variables $\b W$ 
		are unique\-ly determined by the collection of bivariate distributions $(Z_i)_{i\in I}$, $|I| = 2$,
		i.e. by the covariances $Cov(Z_i,Z_j), i \leq j$:
		\begin{enumerate}[(i)]
			\item\label{item:chi2_frailty} Frailty variables with $\chi^2$-distributed marginals: 
			$(W_i)_{i \in I_0} = (Z_i^2)_{i \in I_0}$.
			\item\label{item:log-normal_frailty} Frailty variables with log-normal marginals: $(W_i)_{i \in I_0} = (\exp(Z_i))_{i \in I_0}$.
		\end{enumerate}
	(b) 	The correlated frailty vector $(X_1,\dots,X_d)$ given by \eqref{item:chi2_frailty} or \eqref{item:log-normal_frailty} has independent components
		iff $\bs \Sigma$ has diagonal form.\\
\end{Example}
\noindent The proofs of this example are deferred to \ref{sec:proofs}.

In contrast to Example~\ref{Bsp:doppelt} there exist correlated frailty models with pairwise independent components of $\b X$ but higher order dependence.
For instance let $W_1,W_2,W_3$ be pairwise independent but not totally independent.
Then $S$ has the form $S = \exp \left( - \sum_{i=1}^3 \Lambda_i - \Lambda_{ \{1,2,3\}} \right)$ 
with trivial exponent measures $\Lambda_J = 0$ for $|J| = 2$.
See Section~\ref{sec:higher_dep} for a related discussion which systematically characterises copulas 
having only dependencies of the highest order.

Part (a) of the following example continues Example~\ref{Bsp:doppelt}(a)(i) for the special case $d=3$:
\begin{Example} \label{Bsp:CFM_ZiQuadrat}
(a)	Let $(Z_1,Z_2,Z_3)$ follow a multivariate normal distribution with mean 
	$E (Z_i) = 0$ for $i \in \{1,2,3\}$ and covariance matrix
	\begin{align*}
		\bs \Sigma = 	\begin{pmatrix} 
					\sigma_1^2 & \sigma_{12} & \sigma_{13} 
					\\ \sigma_{12} & \sigma_2^2 & \sigma_{23} 
					\\ \sigma_{13} & \sigma_{23} & \sigma_{3}^2 
				\end{pmatrix}, \quad \sigma_i > 0 \text{ for each } i =1,2,3,
	\end{align*}
	and let $\varrho_{ij} := \frac {\sigma_{ij}} {\sigma_i \sigma_j}$ denote the correlation coefficient of $Z_i$ and $Z_j$.
	Consider the correlated frailty model \eqref{eq:2p2} with frailty variables $W_i = Z_i^2$, $i \in \{1,2,3\}$.
	
	Explicit formulas and their derivation for all dependence parts of the survival function, i.e. for $S_{\{1,2,3\}}, S_{\{i,j\}}, i < j$,
	as well as for the 2-dimensional hazard densities $\lambda_{\{i,j\}}, i < j$, can be found in \ref{sec:formulas_chi_square}.
	There it is seen that the survival copula $S \circ (S_1^{ -1},S_2^{ -1},S_3^{ -1})$ only depends on the bivariate parameters $\varrho_{ij}$.
	This is no surprise as the vector of frailty variables $\b W$ is completely determined by its pairwise covariances.
	
	However, this rather simple dependence structure of $\b W$ involves a more complicated dependence structure for $\b X$:
	the dependence parts of second as well as of third order are non-trivial.
	The discussion of the bivariate version of this example is continued in Example~\ref{ex:formulas_and_graphics}.
\\
	(b) 	Further (mostly bivariate and rather particular) examples for correlated frailty models based on different distributions have been studied and applied. 
		An overview is given in Chapter~5 of \cite{Wienke11}.
		Chapter~6 of his book shows that the copulas of shared and correlated frailty models based on particular 
		distributions can be derived without frailty models. A generalization of the approach is illustrated.
\end{Example}
%
%
%

%
%
%
\section{Correlated frailty models with sum-infinitely divisible scale distributions} 
\label{sec:frailty_id}
In this section we derive exponential dependence measures of correlated frailty models from the well known L\'evy-Khintchine formula which is first summarized for Laplace transforms.
We refer to \cite{Petrov95}  and \cite{Meerschaert01} 
for the notion of sum-infinite divisibility of $\mathfrak{L}(\b W)$ 
and for the L\'evy-Khintchine formula for Fourier transforms.
For $\b W$ with values in $\Rdp$ it is known that the Gaussian part vanishes, 
that its L\'evy measure is concentrated on $\Rdpo0$
and that the integral of $\| \b x \| \b 1\{ 0 < \| \b x \| \leq 1\}$ with respect to the L\'evy measure is bounded.
Here $\| \cdot \|$ denotes the Euclidean norm.
This yields the following 
\begin{Corollary} \label{cor:Levy_Khin}
Let $Q_\b 0 = \mathfrak{L}(\b W)$ be sum-infinitely divisible and concentrated on $[0,\infty)^d$ 
with L\'evy-Khintchine triplet $(\b b, \b 0, \eta_\b 0)$
where $\b b = (b_1,\dots,b_d) \in \Rdp$ and $\eta_{\b 0}$ is a L\'evy measure concentrated on $\Rdpo0$.
Then $Q_\b 0$ is concentrated on $(0,\infty)^d$ 
iff for each coordinate $i=1,\dots,d$ either $b_i > 0$ holds 
or the univariate L\'evy measure of $W_i$ is unbounded.
\end{Corollary}
The proof follows from \cite{Hartman42}.
In case $b_i=0$ they proved that $W_i$ is positive whenever its univariate L\'evy measure is unbounded.
Assume henceforth that $Q_\b 0((0,\infty)^d) = 1$.
In this case, we get a well known extension of the related representation for Fourier transforms to Laplace transforms;
see for instance \cite{Janssen85}.
Lemma~\ref{lemma:surv_eq_lapl}(a) then builds the following connection to the survival function of the correlated frailty model $\b X$ with frailty variable $\b W$:
For all $\b t \in [0,\infty)^d$ we have
\begin{align} \label{eq:Laplace_Khintchine}
	S(\b t) = \psi_{Q_\b 0}(\b t) = \exp \Big( -<\b t,\b b> + \int_{\Rdpo0} \left( e^{-<\b t,\b x>} - 1 \right)
		\tn d \eta_{\b 0}(\b x) \Big).
\end{align}
This particular correlated frailty model is studied throughout this section.

In a remark in~\ref{sec:proofs} it is shown
that, for each correlated frailty model $\b X$ with survival function $S$,
the sum-infinite divisibility of $\mathfrak{L}(\b W)$ 
is equivalent to the minimum-infinite divisibility of $S$.
Thus, we have the following relationship where $\mu$ is the exponent measure of the minimum-infinitely divisible $S$ 
(see Section~\ref{sec:exponent_measures_of_dependence} or \cite{Resnick87}, Section~5.3, for this particular non-copula case):
For all $\b t \in [0, \infty)^d$,
\begin{align*}
 \mu \left((\b t, \bs \infty]^{\mathcal C} \right) \; = \; < \b t, \b b> - \int_{\Rdpo0} \left( e^{-<\b t,\b x>} - 1 \right)
		\tn d \eta_{\b 0}(\b x).
\end{align*}

With the help of equation~\eqref{eq:Laplace_Khintchine} we now express the $|I|$-dimensional hazard rates 
	$\lambda_I, I \subset I_0$,  
in terms of a family of L\'evy measures 
	$\eta_{\b t}$ 
given by
\begin{align*}
	\frac{\tn d \eta_{\b t}}{\tn d \eta_\b 0}(\b x) := \exp(- < \b x, \b t >),\ 
	\b t \in [0,\infty)^d.
\end{align*}
Notice the similarity of this Radon-Nikodym density to the definition of the exponential family $Q_{\b t}$ in~\eqref{Formel2-3}.
\begin{Lemma} \label{lem:lambda_id_w}
	  For all $i \in I \subset I_0$ such that $|I| \geq 2$ we have, for $\b t_I \in (0,\infty)^{|I|}$,
	\begin{align*}
	 	\lambda_i(t_i) = b_i + \int_{\Rdpo0} x_i \tn d \eta_{(t_i,\b 0)}(\b x) \;\; \text{and} \;\;
		 \lambda_I(\b t_I) = \int_{\Rdpo0} \prod_{i \in I} x_i \tn d \eta_{(\b t_I, \b 0)}(\b x) .
	\end{align*}
\end{Lemma}
\noindent The technical proofs of Lemmas~\ref{lem:lambda_id_w} and~\ref{lem:Lambda_id_w} 
are given in \ref{sec:proofs}.
\begin{Remark}
  It is surprising that the marginal hazard rates $\lambda_i$ can be expressed 
  by the first moment of $Q_{(t_i, \b 0)}^{\pi_i}$ (cf.~\eqref{Formel2-6}) as well as
  by the first moment of the L\'evy measure $\eta_{(t_i, \b 0)}^{\pi_i}$ whenever $b_i = 0$ holds:
 \begin{align*}
    \lambda_i(t_i) = \int_{(0,\infty)} x_i \tn d Q_{(t_i,\b 0)}^{\pi_i}(x_i) = \int_{(0,\infty)} x_i \tn d \eta_{(t_i, \b 0)}^{\pi_i}(x_i)
 \end{align*}
 with canonical projection $\pi_i = \pi_{\{i\}}$.
 Notice also that the higher dimensional quantities $\lambda_I$ of Lemma~\ref{lem:lambda_id_w}
 have a pleasant form whereas higher derivatives of $\log S$ are much more complicated 
 when expressed as functions of $Q_{(\b t_I, \b 0)}$.
\end{Remark}
Furthermore, Lemma~\ref{lem:lambda_id_w} enables us to derive the following relationship 
between dependence hazard measures of correlated frailty models and finite L\'evy measures of frailty variables $\b W$.
\begin{Lemma} \label{lem:Lambda_id_w}
  Assume $\eta_\b 0$ to be finite, i.e. $\eta_\b 0$ is the L\'evy measure of a compound Poisson distribution,
  and assume $b_i > 0$ for all $i \in I_0$.
 For all $i \in I \subset I_0, |I| \geq 2,$ we have
 \begin{align*}
   \Lambda_i(t_i) & = b_i t_i + \eta_{\b 0}(\Rdpo0) - \eta_{(t_i,\b 0)} (\Rdpo0) \\
   \text{and} \quad \Lambda_I(\b t_I) & = \sum_{J \subset I} (-1)^{|J|} \eta_{(\b t_J,\b 0)} (\Rdpo0) \quad \text{for all} \; \b t \in \Rdp.
 \end{align*}
 \end{Lemma}
%
%
%
%
%
%
%
%
%
\section{Semiparametric dependence models} \label{sec:semiparametric_dependence_models} 
In this section the exponent dependence measures $\Lambda_I$ are parametrized 
in the copula manner by a semiparametric dependence model given by hazard quantities,
following the bivariate research of \cite{JR02}.
Subsequently, we denote all quantities belonging to functions in the copula case by an additional $0$-index.

Let $\Lambda_I$ be an $|I|$-dimensional signed measure of order $|I|\geq 2$ of some continuously distributed random vector $(X_1,\dots,X_d)$ on $\mathbb R^d$ with the representation
\begin{equation*}
	\Lambda_I (\b t_I) = \int_{-\bs \infty}^{\b t_I} \lambda_I(\b s_I) \intd \b s_I
\end{equation*}
where $\bs \infty=(\infty, \dots, \infty)$ is of dimension $|I|$.
Let $\lambda_i$ be the hazard functions of the univariate $\Lambda_i$.
Then $\Lambda_I$ can be rewritten as function of the marginal hazards
\begin{align} \label{eq:Lambda_I=int}
	\Lambda_I (\b t_I) = \int_{-\bs \infty}^{\b t_I} \gamma_I(\b s_I) \intd \big( \Lambda_{i_1} \otimes \dots \otimes \Lambda_{i_r} \big) (\b s_I)
\end{align}
for $I \subset I_0, \b s_I = (s_i)_{i \in I}$ and the $|I|$-dimensional dependence function 
\begin{align} \label{eq:gamma_I}
	\gamma_I (\b s_I) = \frac {\lambda_I(\b s_I)} {\prod_{i \in I} \lambda_{i}\left(s_{i}\right)}.
\end{align}
In this section we always assume that the densities $\lambda_I$ exist.
It is our aim to split the signed dependence measures $\Lambda_I$ up into marginal effects and a parametrized dependence function
$\gamma_{0,I} : [0,1]^{|I|} \to \mathbb R$, 
which may serve as a parameter of dependence.
For this purpose let 
\begin{align} \label{eq:uni_haz_measure}
	\frac {\intd \Lambda_0} {\intd \lambda \! \! \lambda|_{[0,1]}}(x) = \frac {\b 1_{[0,1]}(x)} {1-x}
\end{align}
be the univariate hazard measure $\Lambda_0$ of the uniform distribution on $[0,1]$. 
By~\eqref{eq:one_and_anoter_one} the univariate hazards can be transformed by 
\begin{align} \label{eq:transformation_formula_hazards}
	\mathfrak{L}(F_i|\Lambda_i) = \Lambda_0 \text{ and } \Lambda_i = \mathfrak{L}\left( F_i^{-1} \mids \Lambda_0 \right)
\end{align}
where %
$F_i$ is the distribution function of $X_i$.
\begin{Lemma}
	(a) Writing the survival function of $(F_1(X_1),\dots,F_d(X_d))$ in the form given by~\eqref{1p9},
	then the exponent measures $\Lambda_{0,I}, I \subset I_0$, given by~\eqref{eq:quasi_Def_von_Lambda_I}
	can be represented via
	\begin{align} \label{eq:Lamda_0I=int}
		\Lambda_{0,I} (\b v_I) = \int_{\b 0}^{\b v_I} \gamma_{0,I} (\b u_I) \intd \Lambda_0^I (\b u_I), \ \b v_I \in [0,1]^{|I|},
	\end{align}
	where
	\begin{align} \label{eq:gamma_0I}
		\gamma_{0,I} : [0,1]^I \to \mathbb R,\ \gamma_{0,I} (\b u_I) := \gamma_I\big( ( F^{-1}_{i} ( u_{i} ) )_{i \in I} \big),\ \b u_I = (u_{i})_{i \in I},
	\end{align}
	and $\Lambda_0^I$ denotes the $|I|$-fold product measure of~\eqref{eq:uni_haz_measure}.
	\\
	(b) Suppose that $S_0$ is the survival function of a copula with $\Lambda_{0,I} = \Lambda_0$ whenever $|I| = 1$ and 
	\begin{align*}
		S_0(\b v) = \exp \Big( \sum_{\emptyset \neq I \subset I_0} (-1)^{|I|} \Lambda_{0,I} (\b v_I) \Big),
	\end{align*}
	where $\Lambda_{0,I}$ is given by a function $\gamma_{0,I}$ via \eqref{eq:Lamda_0I=int}.
	If a $d$-dimensional random variable $(U_1,\dots,U_d)$ has the survival function $S_0$,
	then $\b V = (V_1, \dots, V_d) = \left( F^{-1}_1(U_1), \dots,  F^{-1}_d(U_d) \right)$
	has marginal hazards $\Lambda_i (t_i) = - \log (1-F_i(t_i) ) $.
	The survival function of $\b V$ has the representation 
	\begin{align*}
	  S = \exp \Big( \sum_{\emptyset \neq I \subset I_0} (-1)^{|I|} \Lambda_I \Big),
	\end{align*}
	with $\Lambda_I$ given by~\eqref{eq:Lambda_I=int} and
	$\gamma_I (\b s_I) = \gamma_{0,I} \big( ( F_{i} ( s_{i} ) )_{i \in I} \big), \b s = (s_i)_{i \in I}.$
\end{Lemma}
The proof is based on the transformation formula for hazards \eqref{eq:transformation_formula_hazards} since 
$\mathfrak{L} \left(F_i^{-1} \mids \Lambda_0 \right) = \Lambda_i$.
\begin{Remark}
	(a) The dependence parameter of interest \eqref{eq:gamma_I} is just 
	\begin{align*}
		\frac {\intd \Lambda_I} {\intd ( \otimes_{i\in I} \Lambda_i ) } = \gamma_I
	\end{align*}
	which is completely determined by $\gamma_{0,I}$ \eqref{eq:gamma_0I} via
	\begin{align*}
		\Lambda_I = \int \gamma_{0,I} \left( (F_i)_{i\in I} \right) \intd ( \otimes_{i\in I} \Lambda_i ).
	\end{align*}
	The dependence functions $\gamma_{0,I}$ form a generalization of $\gamma_{d,0}$ of \cite{JR02}
	which also may be utilized for the construction of dependence tests in semiparametric contexts.
	\\
	(b) In survival analysis often the componentwise minimum $\min(\b X,\b X')$ of independent vectors $\b X$ and $\b X'$ is considered.
	In this case the exponent measures of $\b X$ and $\b X'$ sum up.
	\\
	(c) Another useful transformation is given by 
	\begin{align*}
		\b Y = (Y_1,\dots,Y_d) := ( \Lambda_1(X_1), \dots, \Lambda_d(X_d) ),
	\end{align*}
	whose marginals $Y_i$ are standard exponentially distributed.
	In this case we have $\frac{\intd \Lambda_i^{(\b Y)}}{\intd \lambda \!\! \lambda} = \b 1_{(0,\infty)}$ for the marginal hazard measures $\Lambda_i^{(\b Y)}$ of $\b Y$ and
	\begin{align*}
		\gamma_I^{(\b Y)}(\b s_I) = \lambda_I^{(\b Y)}(\b s_I) = \gamma_{0,I} \left( \left( 1 - \exp (-s_i) \right)_{i\in I} \right).
	\end{align*}
	(d) In the bivariate case $d=2$ we recall the nice interpretation of $\lambda_{\{1,2\}}$ as local dependence parameter with
	\begin{align} \label{eq:gamma_12}
		\gamma_{\{1,2\}} (s_1,s_2) 
			= \frac {\lambda_{\{1,2\}}(s_1,s_2)} {\lambda_1(s_1) \lambda_2(s_2)} 
			= \frac {\intd \Lambda_{\{1,2\}}} {\intd ( \Lambda_1 \otimes \Lambda_2 )} (s_1,s_2).
	\end{align}
\end{Remark}
\begin{Example}\label{Bsp:The_parametrization}
	A special case of the parametrization \eqref{eq:gamma_12} is the so-called proportional hazard dependence given by the survival function
	\begin{align} \label{eq:414}
		S_{\text{prop}} (t_1,t_2) = \exp( - \Lambda_1 (t_1) - \Lambda_2(t_2) - \beta \Lambda_1(t_1) \Lambda_2(t_2) ),\ 0\leq \beta \leq 1,
	\end{align}
	which is the well known bivariate exponential survival function if the marginals are exponential distributions with $\Lambda_i(t_i) = t_i$ for $t_i > 0$.
	Note that
	\begin{align} \label{eq:415}
		\Lambda_{\{1,2\}}^{\text{prop}} = - \beta \ \Lambda_1 \otimes \Lambda_2
	\end{align}
	holds.
	That is, the local dependence parameter~\eqref{eq:gamma_12} is the constant function $-\beta$.
	Therefore, \eqref{eq:415} is called the proportional hazard rate dependence model 
	and it can be considered as dependence counterpart of Cox models;
	see~\cite{JR02} for dependence tests.
	A visualization of $\gamma_{\{1,2\}}$ for $\beta = 1$ is given in Figure~\ref{fig:gamma_0_fuer_prop}
	of the appendix.
\end{Example}
Subsequently, the proportional hazard rate dependence model~\eqref{eq:415} with measure $\Lambda_1 \otimes \Lambda_2$ serves as a benchmark
and we consider $\gamma_{\{1,2\}}$ of \eqref{eq:gamma_12} having the interpretation \eqref{eq:interpretation_lambda_12_local_dep_hazard} in mind.
For the visualization, below we present some bivariate contour plots of the normalized dependence function $\gamma_{0,\{1,2\}}$ on $[0,1]^2$;
see~\eqref{eq:gamma_0I}.
A selection of plots is presented in Figures~\ref{fig:gamma_0_fuer_clayton}--~\ref{fig:gamma_0_fuer_Zj2} whereas further plots (also for different parameters)
can be found in \ref{sec:graphics}.
\begin{Example} \label{ex:formulas_and_graphics}
	(a) (Binary correlated frailty model) Consider the correlated frailty model and a binary index set $I = \{i,j\}$, $|I| = 2 \leq d$.
	The dependence measure \eqref{eq:Lambda_I=int} is then given by the exponential family \eqref{Formel2-3} with
	\begin{align*}
		\gamma_I (t_i,t_j) = \frac {Cov_{Q_{(\b t_I, \b 0)}} (W_i,W_j)} {E_{Q_{(t_i, \b 0)}}(W_i) E_{Q_{(t_j, \b 0)}}(W_j)}.
	\end{align*}
	(b) The Clayton copula is an Archimedean copula with generator 
		$\varphi(s) = \frac 1 {1+s}$; see Example~\ref{example:archimedean_copulas}.
		It is also derived as a classical shared frailty model with standard exponentially distributed frailty variable $W_1 = W_2$. 
		We see that this is an exponential scale model on the diagonal set 
		$\{(x,x) : x \geq 0\} \subset \mathbb R^2$ 
		with 
		$\lambda_i(s) = \frac 1 {1+s}$, 
		$\lambda_{\{1,2\}}(t_1,t_2) = (1+t_1+t_2)^{-2}$,
		$\gamma_{\{1,2\}}(t_1,t_2) = \frac {(1+t_1) (1+t_2)} {(1+t_1+t_2)^2}$ 
		and
		\begin{align*}
			\gamma_{0,\{1,2\}} (u_1,u_2) 
			= \frac 
				{\big(\frac 1 {1-u_1}\big) \big(\frac 1 {1-u_2}\big)} 
				{\big(1 + \frac {u_1} {1-u_1} + \frac {u_2} {1-u_2} \big)^2}
			= \frac{(1-u_1)(1-u_2)}{(1-u_1 u_2)^2},\ 
				0 \leq u_1, u_2 < 1.
		\end{align*}
		A visualization of $\gamma_{0,\{1,2\}}$ is given in Figure~\ref{fig:gamma_0_fuer_clayton}.
		\begin{figure}[H]
		  \centering
		  \mbox{ 
		    \begin{subfigure}{0.4\textwidth}
			\includegraphics[width=\textwidth]{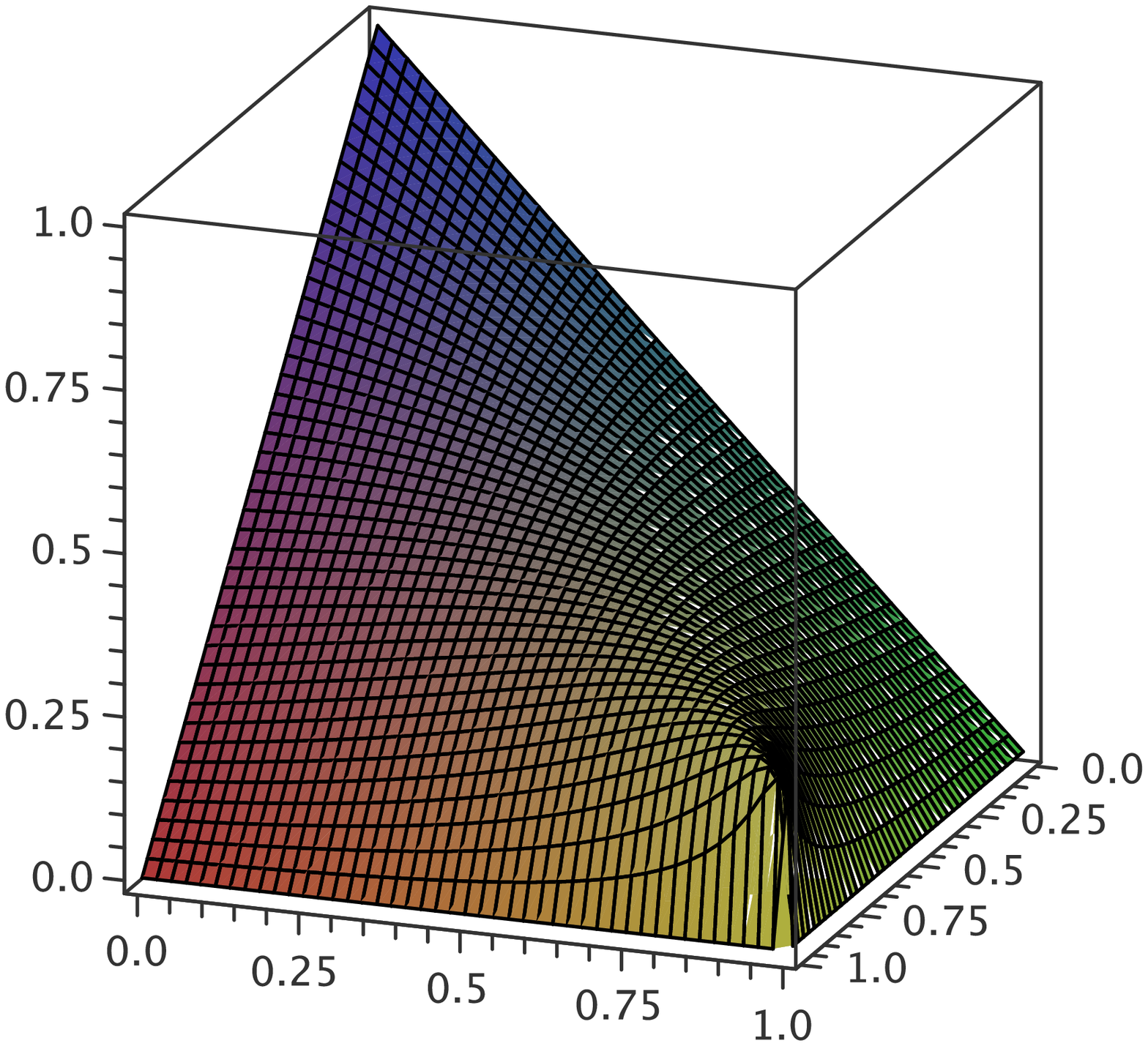}
		    \end{subfigure} 
		    \begin{subfigure}{0.4\textwidth}
			\includegraphics[width=\textwidth]{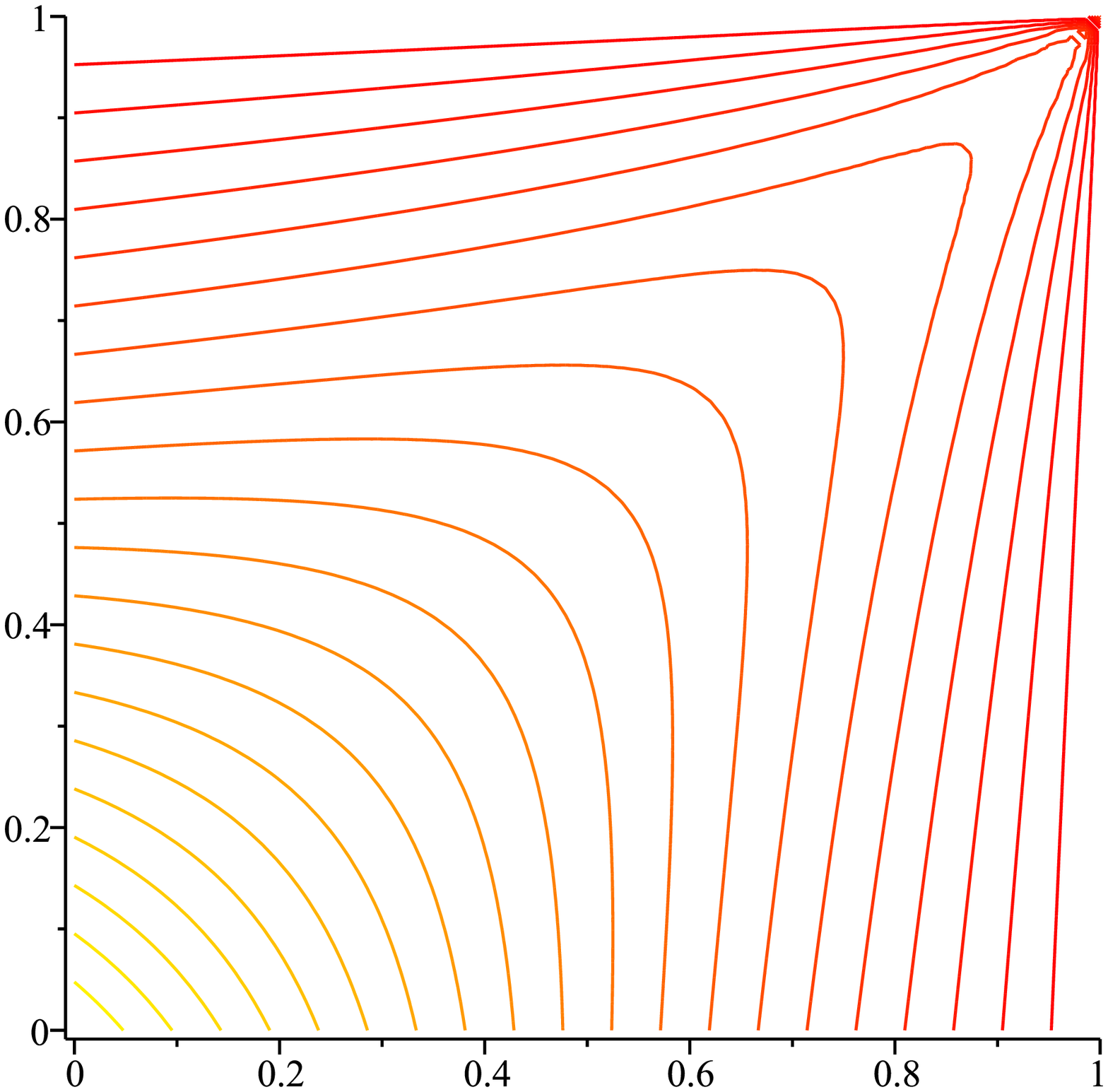}
		    \end{subfigure}
		  }
		  \caption{Plot and contour curves for the dependence hazard rate derivative $\gamma_{0,\{1,2\}}$ 
			of a bivariate shared frailty model based on a standard exponential distribution $\mathfrak{L}(W)$ over the range $[0,1)^2$.
			This choice of distributions results in a Clayton copula.}
		  \label{fig:gamma_0_fuer_clayton}
		\end{figure}
	(c) The Frank copula is an Archimedean copula with generator 
		\begin{align*}
			\varphi(s) = - \log\Big( \frac { \exp(-\theta s) - 1 } { \exp(-\theta) - 1 } \Big),\ \theta \in \mathbb R\setminus\{0\};
		\end{align*}
		see Example~\ref{example:archimedean_copulas}.
		For a bivariate survival function with the Frank copula as survival copula we obtain
		\begin{align*}
			\lambda_{\{1,2\}}\left(F^{-1}_1(u_1),F^{-1}_2(u_2) \right) = 		\theta^2\ 
														\frac {e^{\theta(u_1-1)} e^{\theta(u_2-1)}} {e^{-\theta} -1 } 
														\frac { \log(1+ g) - g } { [(1+g) \log(1+g)]^2 },
		\end{align*}
		where
		$g := g(u_1,u_2,\theta) := \left(e^{\theta(u_1-1)}-1\right)  \left(e^{\theta(u_2-1)}-1\right) / (e^{-\theta} -1).$
		By \eqref{eq:gamma_0I} and \eqref{eq:gamma_12} we have
		\begin{align*}
			\gamma_{0,\{1,2\}} (u_1,u_2) = \lambda_{\{1,2\}}\left(F^{-1}_1(u_1),F^{-1}_2(u_2) \right) (1-u_1) (1-u_2)
		\end{align*}
		as $\lambda_i\left(F^{-1}_i(u)\right) = \frac 1 {1-u}$ for univariate hazard rates $\lambda_i$ with corresponding distribution functions $F_i$.
		A visualization of $\gamma_{0,\{1,2\}}$ is given in Figure~\ref{fig:gamma_0_fuer_frank}.
	\begin{figure}[H]
	  \centering
	  \mbox{ 
		\begin{subfigure}{0.4\textwidth}
			\includegraphics[width=\textwidth]{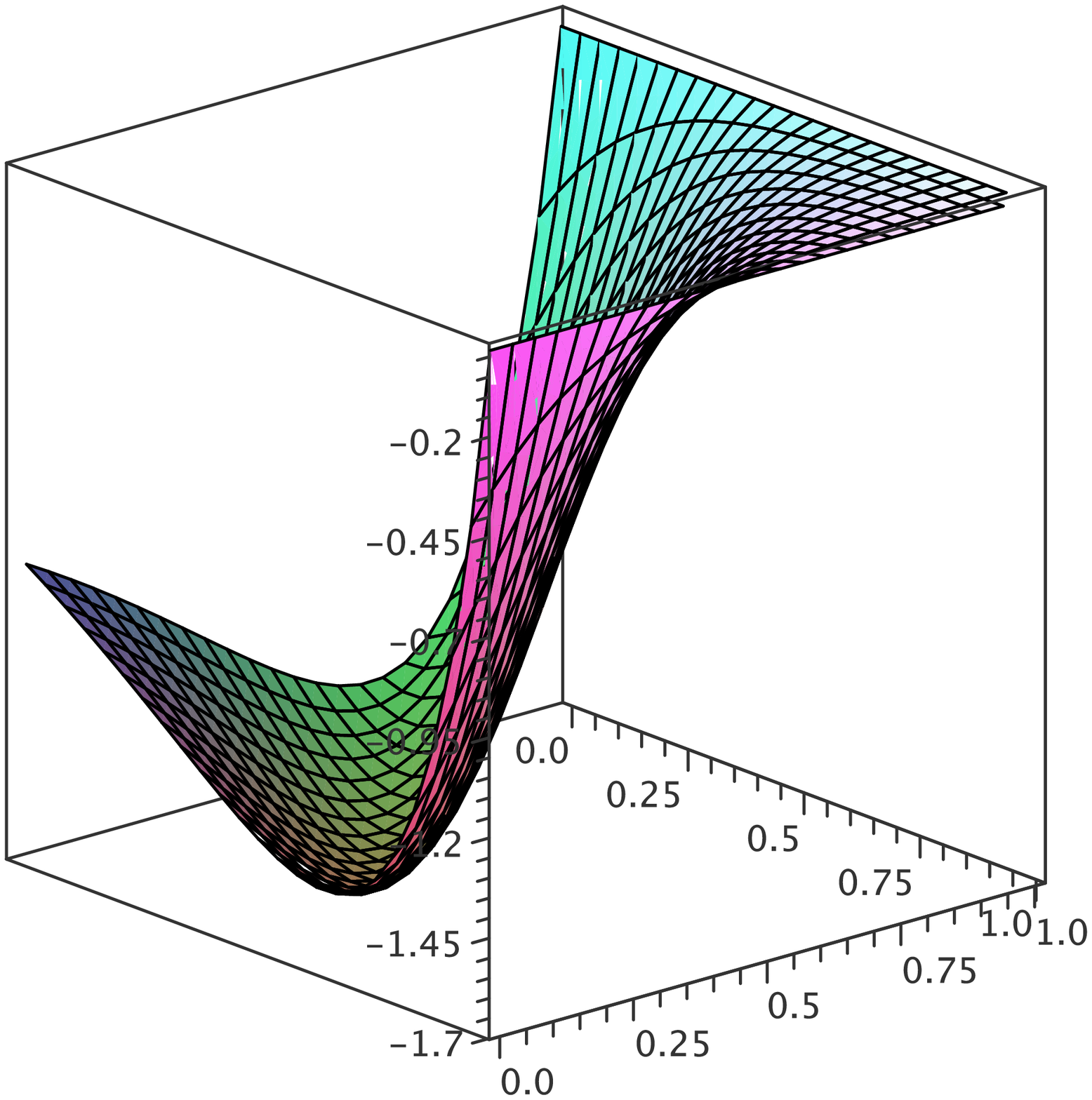}
			\caption{$\theta=-5$}
		\end{subfigure} 
		\begin{subfigure}{0.4\textwidth}
			\includegraphics[width=\textwidth]{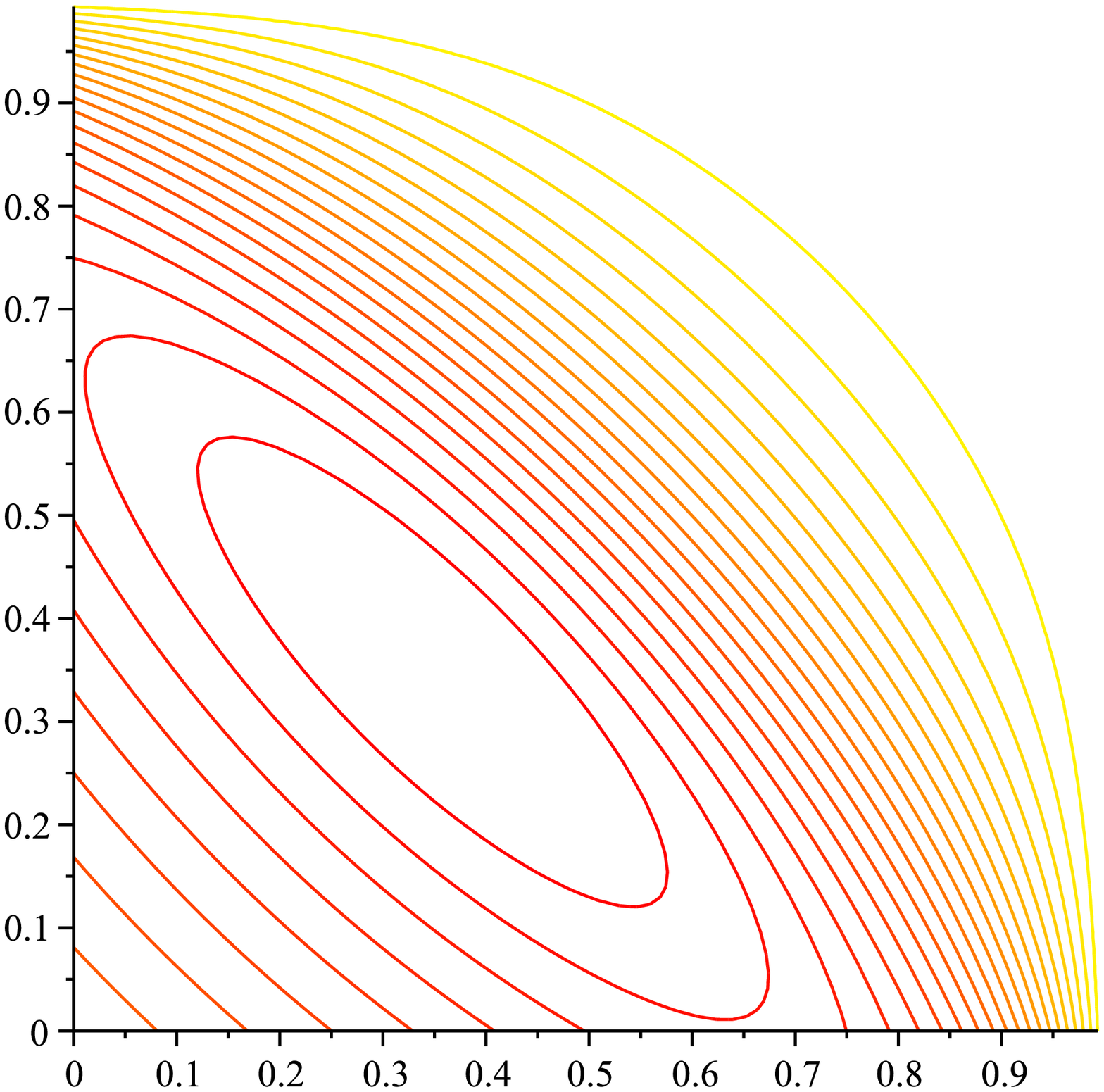}
			\caption{$\theta=-5$}
		\end{subfigure}
	  } \\
	  \mbox{ 
		\begin{subfigure}{0.4\textwidth}
			\includegraphics[width=\textwidth]{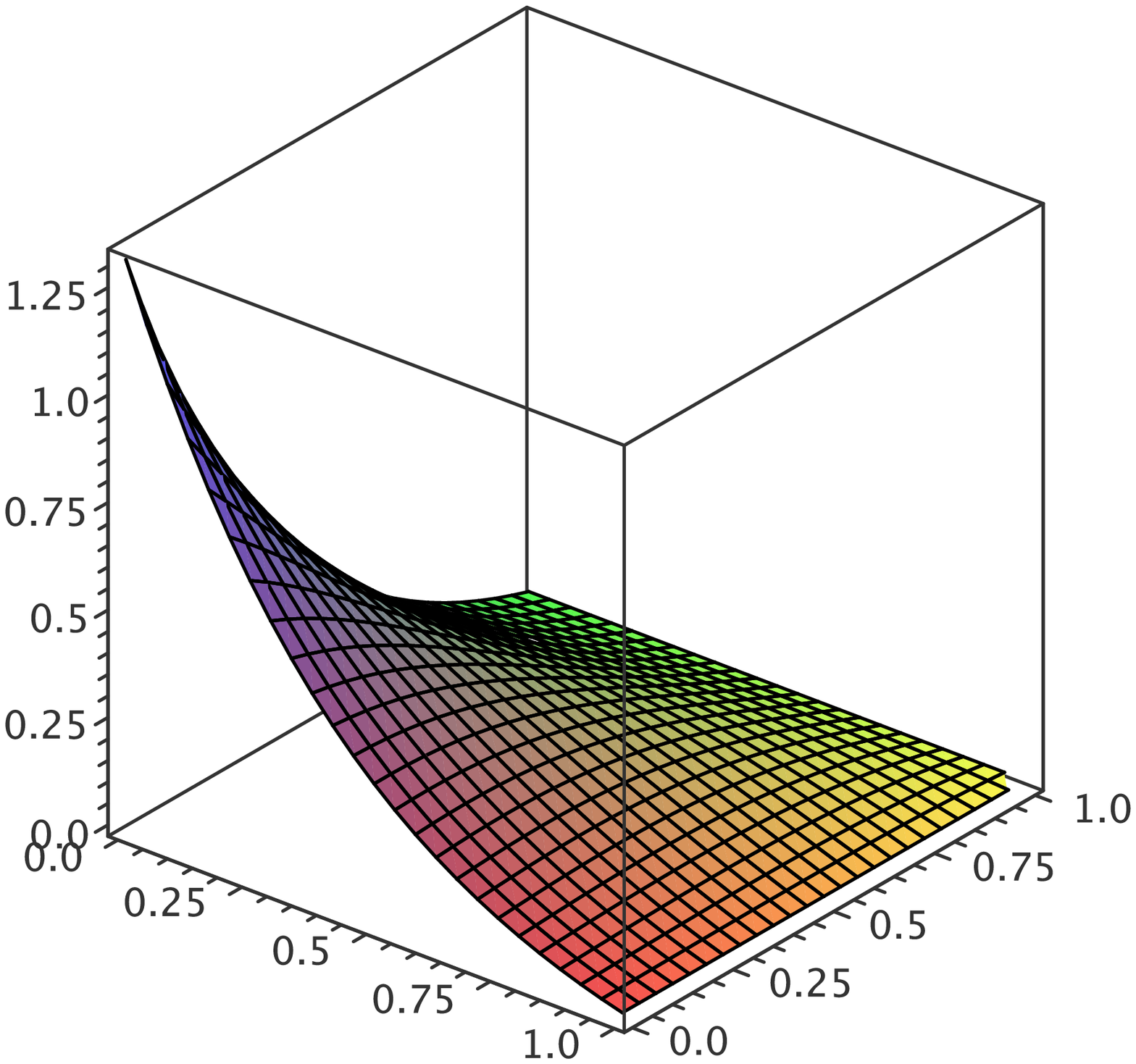}
			\caption{$\theta=2$}
		\end{subfigure} 
		\begin{subfigure}{0.4\textwidth}
			\includegraphics[width=\textwidth]{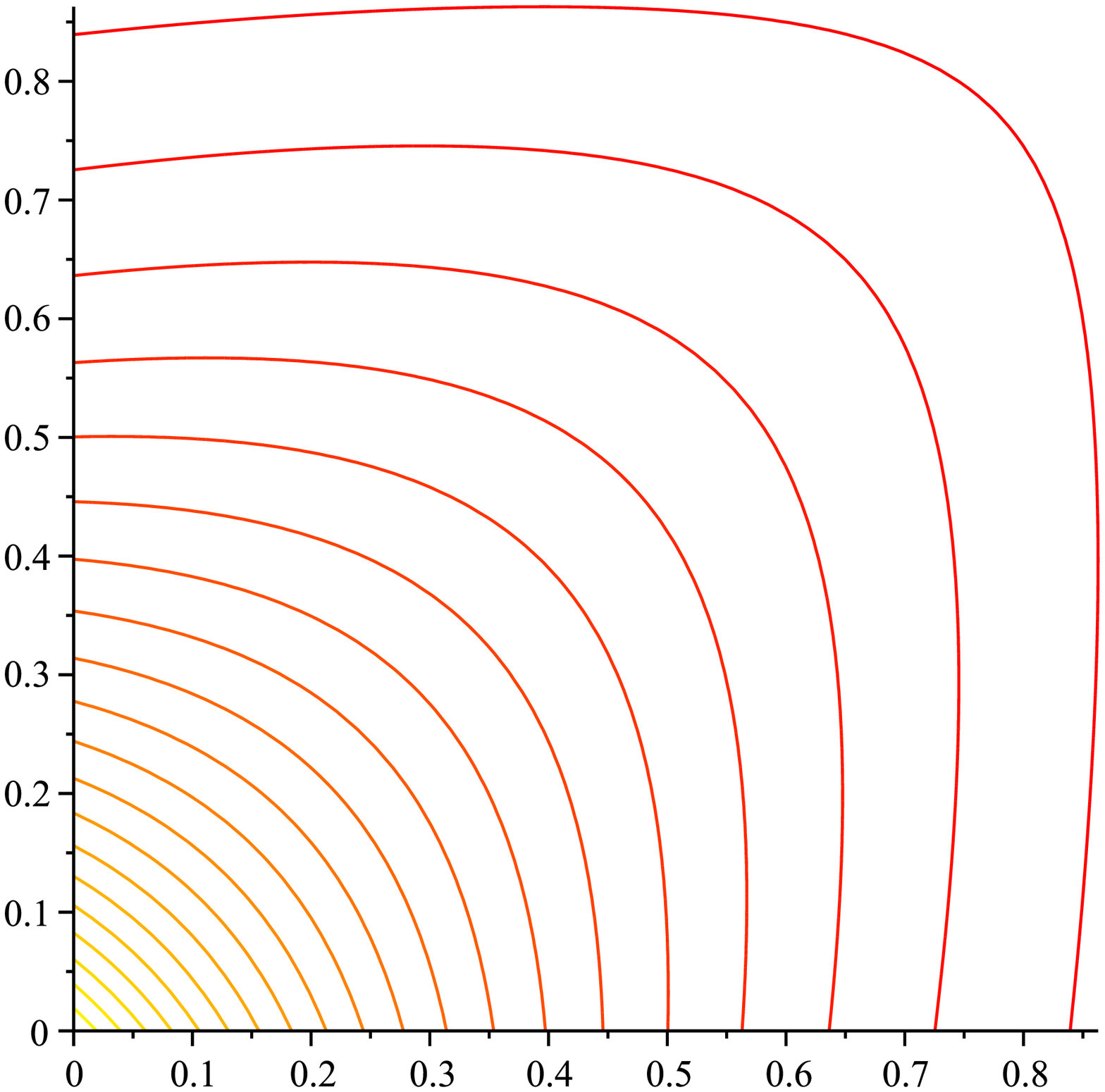}
			\caption{$\theta=2$}
		\end{subfigure}
	  }
	  \caption{Plots and contour curves for the dependence hazard rate derivative $\gamma_{0,\{i,j\}}$ of a model with Frank copula as survival colula over the range $[0,1]^2$, with different values of $\theta$.}
	  \label{fig:gamma_0_fuer_frank}
	\end{figure}
	(d) Consider a shared frailty model based on an inverse Gaussian distribution, i.e. $W_1 = W_2 = Z^{-2}$
	where $Z$ is normally distributed with $E(Z)=0$ and $Var(Z) = \frac{2}{\theta^2} > 0$.
	The Laplace transform of the frailty variable $W$ and the joint survival function of the shared frailty model are given by
		$\psi(t) = \exp( - \theta \sqrt{t})$ and 
		$S(\b t) = \exp( - \theta \sqrt{t_1 + t_2})$, respectively.
		We also see that
	\begin{align*}
		\gamma_{0,\{1,2\}} (u_1,u_2)  = \frac { \log(1-u_1) \log(1-u_2) } { \left\{ \left[\log(1-u_1)\right]^2 + \left[\log(1-u_2)\right]^2 \right\}^{3/2} }
	\end{align*}
	holds, so that this semiparametric dependence function is even invariant under different values of the parameter $\theta > 0$. 
	A visualization of $\gamma_{0,\{1,2\}}$ is given in Figure~\ref{fig:gamma_0_fuer_inverse_gauss}.
	\begin{figure}[H]
	  \centering
	  \mbox{ 
		\begin{subfigure}{0.4\textwidth}
			\includegraphics[width=\textwidth]{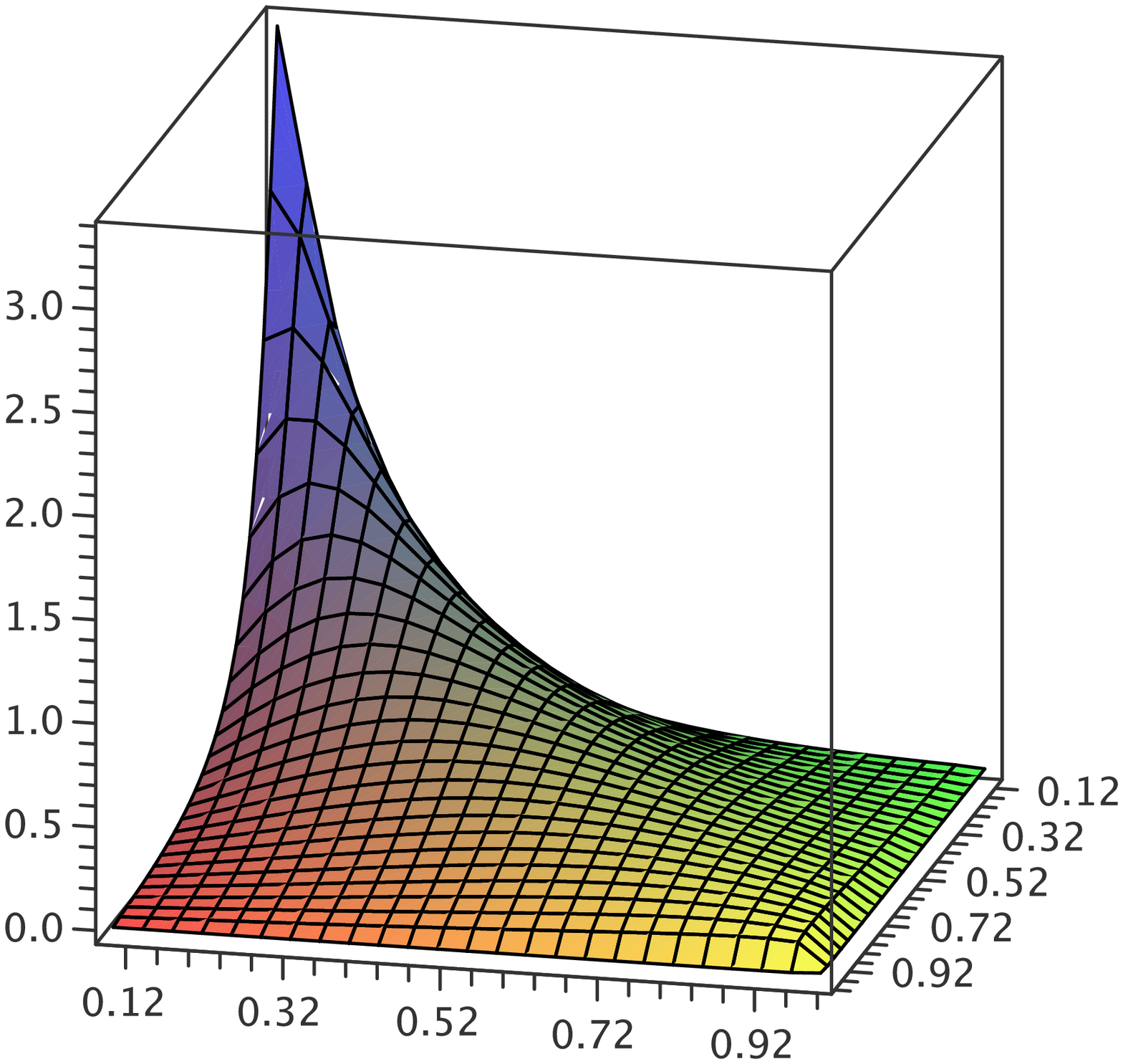}

		\end{subfigure} 
		\begin{subfigure}{0.4\textwidth}
			\includegraphics[width=\textwidth]{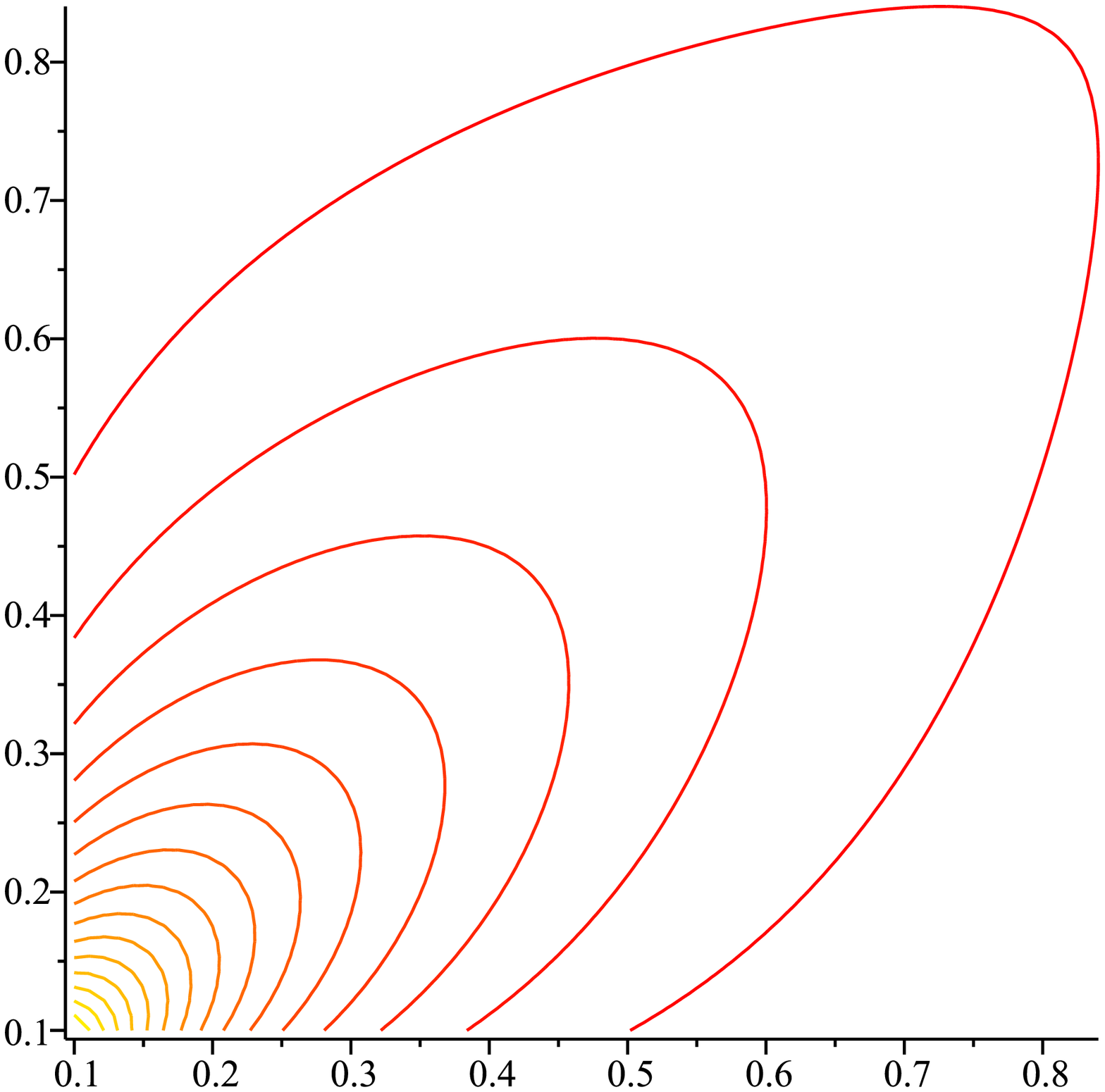}
		\end{subfigure}
	  }
	  \caption{Plot and contour curves for the dependence hazard rate derivative $\gamma_{0,\{1,2\}}$ 
			of a bivariate shared frailty model based on an inverse Gaussian distribution over the range $(\frac 1 {10}, 1]^2$.
			Take note of the pole in the origin $(u_1,u_2) = (0,0)$.}
	  \label{fig:gamma_0_fuer_inverse_gauss}
	\end{figure}
	(e) For the correlated frailty model based on a $\chi^2$-distribution as seen in Example~\ref{Bsp:CFM_ZiQuadrat} we have
	\begin{align*}
		\gamma_{0,\{i,j\}} (u_i,u_j)  = 2 \varrho_{ij}^2 \Big[ \frac {(1-u_i)(1-u_j)} {1-\varrho_{ij}^2 (2- u_i)(2-u_j)u_i u_j} \Big]^2.
	\end{align*}
	Figure~\ref{fig:gamma_0_fuer_Zj2} shows a plot and a contour curve for $\gamma_{0,\{i,j\}}$ with $\varrho_{ij}^2 = 0.9$. 
	Note, that $\gamma_{0,\{i,j\}} (0,0) = 2 \varrho_{ij}^2$.
	Comparing Figure~\ref{fig:gamma_0_fuer_Zj2} with Figure~\ref{fig:gamma_0_fuer_Zj2_2} in \ref{sec:graphics},
	we observe a high early dependence for all values of $\varrho_{ij}^2$ and an increased dependence along the diagonal 
	$\{(u_i,u_j)\in[0,1]^2 : u_i=u_j\}$ for higher values of $\varrho_{ij}^2$.
	\begin{figure}[H]
	  \centering
	  \mbox{ 
		\begin{subfigure}{0.4\textwidth}
			\includegraphics[width=\textwidth]{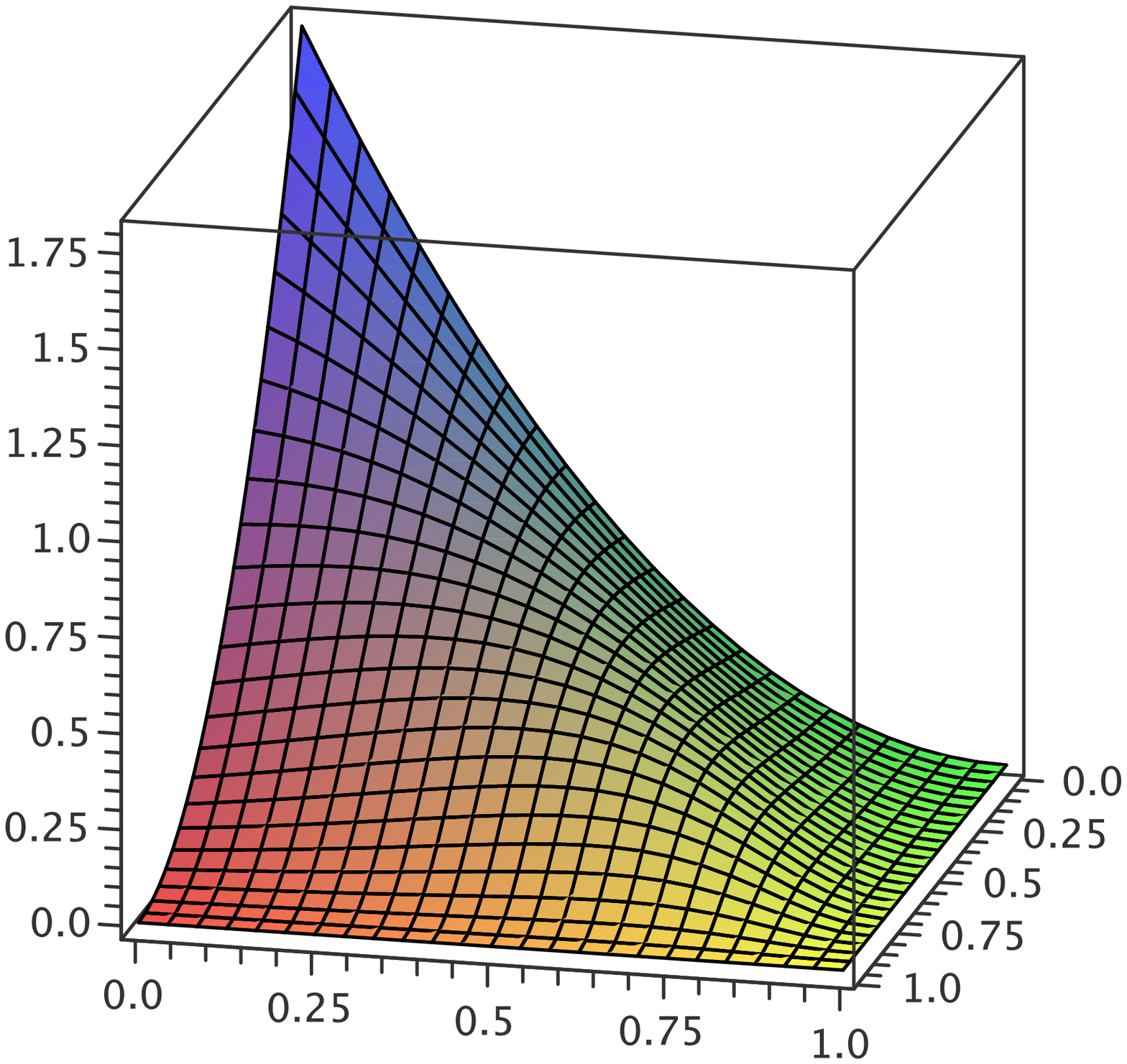}
			\caption{$\varrho_{ij}^2=0.9$}
			\label{fig_gamma:subfig5}
		\end{subfigure} 
		\begin{subfigure}{0.4\textwidth}
			\includegraphics[width=\textwidth]{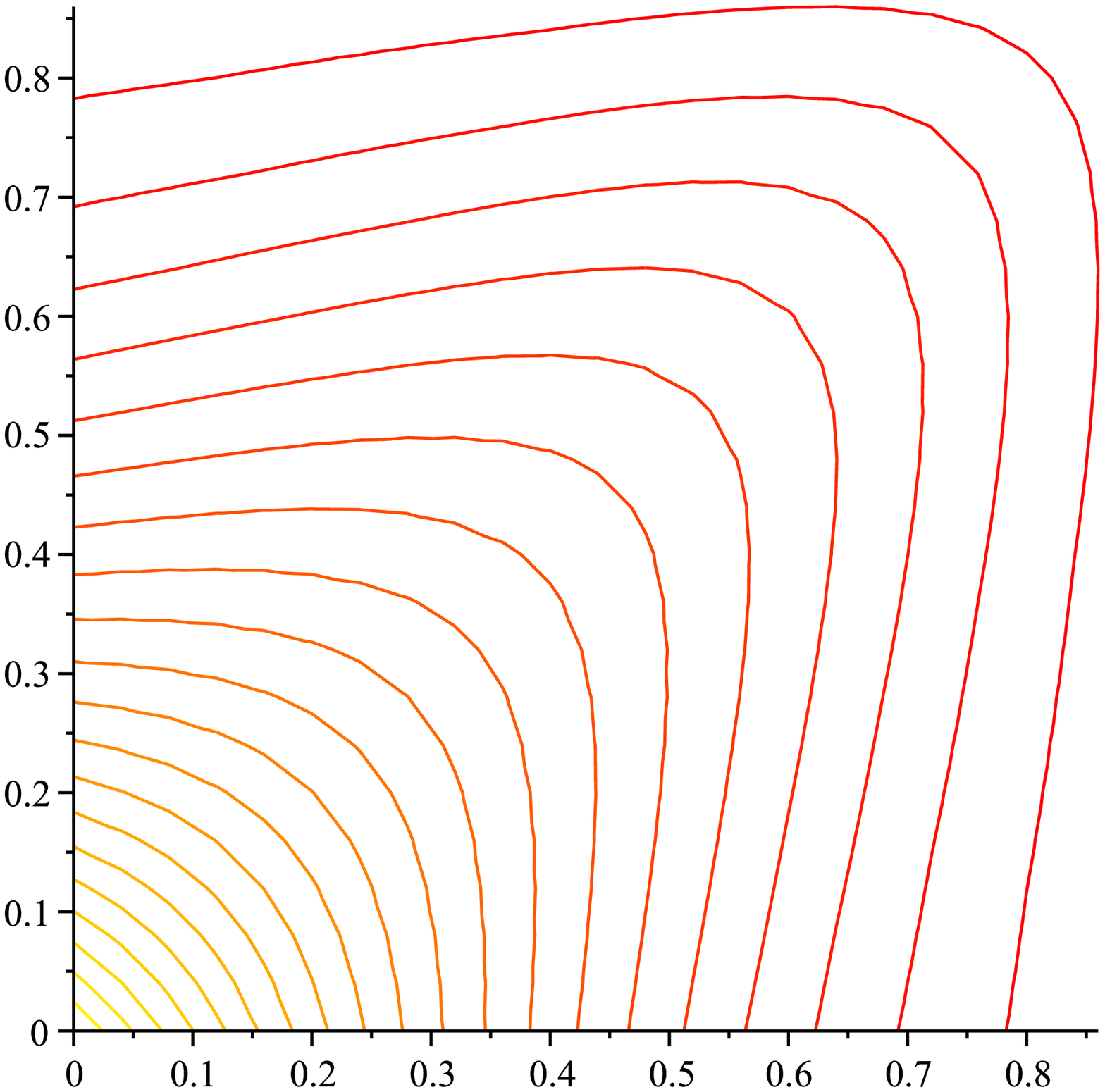}
			\caption{$\varrho_{ij}^2=0.9$}
			\label{fig_gamma:subfig6}
		\end{subfigure}
	  }
	  \caption{Plot and contour curves for the dependence hazard rate derivative $\gamma_{0,\{i,j\}}$ of a bivariate correlated frailty model 
	  based on a bivariate chi-squared distribution over the range $[0,1]^2$, with $\varrho_{ij}^2 = 0.9$.}
	  \label{fig:gamma_0_fuer_Zj2}
	\end{figure}
\end{Example}
%
%
%
%
%
%
%
\section{Copulas with higher degree of dependence: the hazard approach}
\label{sec:higher_dep}
In this section higher order exponent measures are studied in terms of hazard parameters.
We start with binary proportional dependence models which 
are developed in the spirit of Example~\ref{Bsp:The_parametrization}.
\begin{Example}\label{Bsp:higher_dim_prop_dep_model}
	The bivariate proportional dependence model can be extended for dimension $d\geq 2$.
	Consider a family $(\beta_I)_{I\subset I_0, |I| = 2}$ so that 
	$0 \leq \beta_I \leq \frac 1 {(d-1)^2}$ 
	holds for each $I$.
	Let $\Lambda_1,\dots,\Lambda_d$ be continuous univariate hazard measures on $\mathbb R$.
	Then
	\begin{align*}
		S(t_1,\dots,t_d) = 	\exp \bigg( - \sum_{i=1}^d \Lambda_i(t_i) - 
					\sum_{\substack{ i,j=1,\\ i\neq j}}^d \beta_{\{i,j\}} \Lambda_i(t_i) \Lambda_j(t_j) \bigg)
	\end{align*}
	defines the survival function of a $d$-dimensional variable with proportional hazard dependence of the bivariate exponential measures
	$\frac {\intd \Lambda_{\{i,j\}}} {\intd (\Lambda_i \otimes \Lambda_j)} = - \beta_{\{i,j\}}$.
\end{Example}
In Example~\ref{Bsp:higher_dim_prop_dep_model} all higher dimensional exponent measures $\Lambda_I$, $|I|\geq 3$, vanish. 
To see that $S$ is indeed a survival function, assume (with no loss of generality) that $\Lambda_i(t_i) = t_i$ 
and consider the bivariate exponentially distributed random variable $\b X_I$ with survival function 
\begin{align*}
	S^I(t_i,t_j) = \exp \left( - \frac {t_i + t_j} {d-1} - \beta_I t_i t_j \right),
\end{align*}
where $\frac 1 {d-1} \b X_I$ has the form \eqref{eq:414} with parameter $0\leq \beta_I (d-1)^2 \leq 1$.
Introduce now a new $(-\infty,\infty]^d$-valued random variable $\widetilde{\b X}_I$ with the coordinates $\b X_I$ at the position $I$.
Let all other coordinates of $\widetilde{\b X}_I$ be infinite $\infty$.
Thus $\widetilde{\b X}_I$ has the improper survival function
$\widetilde{S}^I(t_1,\dots,t_d) = S^I(t_i,t_j), I= \{i,j\}.$
Consider now independent random variables $(\widetilde{\b X}_I)_I$ for $I\subset I_0, |I| = 2$.
Then $\min_{I, |I|=2} \widetilde{\b X}_I$ has the desired survival function $\prod_{I, |I|=2} \widetilde{S}^I$.

Notice that this modeling of dependence can be generalized by linking arbitrary improper survival functions of the above kind.
This is again accomplished by applying the minimum operation to the corresponding random variables with values in $(-\infty,\infty]^d$.

In contrast to Example~\ref{Bsp:higher_dim_prop_dep_model} survival models with trivial dependence measures $\Lambda_I = 0$ 
for all $I\neq I_0$ and $|I|\geq 2$ are studied below.
Our apporach is based on statistical arguments used earlier for hazard-based score functions.
To this end, let $P_0$ be the uniform distribution on the interval $(0,1)$ with hazard measure $\Lambda_0$.
Introduce the subset $L_{2,d}^{(0)}(P_0^d)$ of those $d$-dimensional, square-integrable functions $g$ in $L_{2,d}(P_0^d)$ for which
\begin{align*}
	\int_0^1 g(x_1,\dots,x_d) \intd x_i = 0
\end{align*}
holds for all $i \in I_0$. Functions $g$ of this kind serve as score functions for statistical models. 
We start with the following useful observation.
Consider a copula which admits a $P_0^d$-density $f: (0,1)^d \rightarrow [0,\infty)$.
Then the following conditions \eqref{item:all_bla} and \eqref{item:the_fun} are equivalent.
\begin{enumerate}[(I)]
	\item\label{item:all_bla} All $|I|$-dimensional marginals of the copula are independence copulas for $I \neq I_0$.
	\item\label{item:the_fun} The function $g:=f-1$ satisfies $g\geq -1$ and $g\in L_{2,d}^{(0)}(P_0^d)$.
\end{enumerate}
In case of \eqref{item:all_bla}, \eqref{item:the_fun} all exponent measures $\Lambda_I$ vanish for $I\neq I_0$, $|I|\geq 2$.
Then the dependence part of $S(\b x) = \prod_{i=1}^d (1 - x_i) S_{I_0}(\b x)$ reads as
\begin{align*}
	S_{I_0}(\b x) = 1 + \frac {\int_{\b x}^{\b1}  g \intd P_0^d} {\prod_{i=1}^d (1-x_i)}
\end{align*}
which is studied below.
Recall that for $d=1$ the operator
\begin{align*}
	R: L_{2,1}^{(0)}(P_0) \to L_{2,1}(P_0),\  R(g)(x) = g(x) -\frac {\int_x^1 g(u) \intd u} {1-x}
\end{align*}
is an isometry between Hilbert spaces; see \cite{Ritov88} as well as \cite{Efron90}.
Note that $R(g)$ has an interpretation as hazard rate derivative for survival models.
We will now introduce a multivariate version $R_d$ of $R$ on $L_{2,d}^{(0)}(P_0^d)$.
\begin{Lemma}
	For $g_1,\dots,g_d \in  L_{2,1}^{(0)}(P_0)$ introduce the multivariate function 
	$g(x_1,\dots,x_d) = \prod_{i=1}^d g_i(x_i)$
	of  $L_{2,d}^{(0)}(P_0^d)$.
	We may then define 
	\begin{align} \label{eq:def_R_d}
		R_d(g)(\b x) := \prod_{i=1}^d R(g_i)(x_i), \ \b x \in (0,1)^d.
	\end{align}
	(a) The operator $R_d$ can be uniquely extended on
	$ L_{2,d}^{(0)}(P_0^d)$
	and
	$R_d:  L_{2,d}^{(0)}(P_0^d) \to L_{2,d}(P_0^d)$
	is an isometry between Hilbert spaces.
	\\
	(b) Suppose that $\gamma = R_d(g)$ holds for some $g\in  L_{2,d}^{(0)}(P_0^d)$.
	Then 
	\begin{align} \label{eq:hmm}
		\frac {\int_{\b x}^{\b 1} g \intd P_0^d} {\prod_{i=1}^d (1-x_i)} = (-1)^d \int_{\b 0}^{\b x} \gamma \intd \Lambda_0^d
	\end{align}
	holds for all $\b x \in (0,1)^d$.
\end{Lemma}
\begin{proof}
		(a) 	The proof is mostly left to the reader and extends the bivariate calculus of \cite{JR02}.
			Let $(\gamma_j)_{j\in\mathbb N}$ be an orthonormal basis of the Hilbert space $L_{2,1}(P_0)$.
			Then product functions 
			$\b x\mapsto \prod_{j=1}^d \gamma_{i_j} (x_{j})$ 
			form an orthonormal basis of $L_{2,d}(P_0^d)$.
			We first introduce $R_d^{-1}$ of these elements.
			In the next step products of $g_j := R^{-1}(\gamma_j)$ are considered as in \eqref{eq:def_R_d}
			and $R_d$ as well as $R_d^{-1}$ can be extended by taking linear combinations of basis elements.
			It is easy to see that $R_d^{-1}$ is surjective.
			If we add $g_0 \equiv 1$ then $(g_j)_{j \in \mathbb N_0}$ is an orthonormal basis of $L_{2,1}(P_0)$ and
			$\b x \mapsto \prod_{j=1}^d g_{i_j} (x_j)$ including the index $0$ is an orthonormal basis of $L_{2,d}(P_0^d)$.
			However, if one index is equal to zero then $\prod_{j=1}^d g_{i_j}$ is orthogonal to $L_{2,d}^{(0)}(P_0^d)$.
			All other elements of this kind are images of $R_d^{-1}$.
			Thus $R_d^{-1}$ is surjective. \\
		(b) 	For $d=1$ formula \eqref{eq:hmm} is proved in (3.8) of \cite{Janssen94}.
			Fubini's theorem extends \eqref{eq:hmm} to $d$-dimensional product functions which are dense
			in $L_{2,d}^{(0)}(P_0^d)$.
			Since $R_d$ is an isometry, statement \eqref{eq:hmm} holds in its general form.
\end{proof}
\begin{Theorem}
	Suppose that \eqref{item:all_bla} or \eqref{item:the_fun} holds and let $S_{I_0}$ be the dependence part of the survival function
	$S(\b x) = \prod_{i=1}^d (1-x_i) S_{I_0}(\b x)$ of the copula.
		\\(a) 	Set $\gamma = R_d(g)$ for $g=f-1$. The dependence part $S_{I_0}$ of $S$ is given by
			\begin{align} \label{eq:S_d=1+}
				S_{I_0}(\b x) = 1 + (-1)^d \int_{\b 0}^{\b x} \gamma \intd \Lambda_0^d, \ \b x\in (0,1)^d.
			\end{align}
		\\	%
		(b) 	For each $0 \leq \vartheta \leq 1$ the function $f_\vartheta := 1 + \vartheta g$ is
			the density of a copula with survival function
			$S_\vartheta(\b x) = \prod_{i=1}^d (1-x_i) S_{I_0,\vartheta} (\b x)$.
			The dependence part $S_{I_0,\vartheta}$ is given by
			\begin{align} \label{eq:S_d,theta=exp}
				S_{I_0,\vartheta}(x) = \exp\left( (-1)^d \vartheta \int_{\b 0}^{\b x} \gamma \intd \Lambda_0^d + o(\vartheta) \right)
			\end{align}
			as $\vartheta \to 0$ with a uniform remainder $o(\vartheta)$ on $[0, 1-\epsilon]^d$ for $\epsilon > 0$.
\end{Theorem}
\begin{proof}
		(a) 	Observe that $S(\b x) = \prod_{i=1}^d (1-x_i) + \int_{\b x}^{\b 1} g \intd P_0^d$ holds.
			If we divide by the marginals, then \eqref{eq:hmm} yields the result.
		\\(b) 	Observe that $\vartheta \gamma$ corresponds to the density of \eqref{eq:S_d=1+}.
			Taking the logarithm of \eqref{eq:S_d=1+} we obtain \eqref{eq:S_d,theta=exp}.
\end{proof}
\begin{Remark}
		(a) 	Let $S_{I_0,\vartheta} = \exp\left( (-1)^d \Lambda_{I_0,\vartheta} \right)$ be the 
			exponent representation of \eqref{eq:S_d,theta=exp}.
			Up to the term $o(\vartheta)$ the exponent is related to 
			$\vartheta \int_0^x \gamma \intd \Lambda_0^d$.
		\\(b) 	Note that $g := \frac \intd {\intd \vartheta} \log f_\vartheta |_{\vartheta = 0}$
			is the score function of the densities $f_\vartheta$ at $\vartheta = 0$.
			These score functions can be used to obtain score tests for the null hypothesis of independence,
			i.e. $\vartheta = 0$, against higher degree dependence.
\end{Remark}
Statistical aspects of hazard dependence models will be studied in a forthcoming work.
%
%
%
%
\section{Discussion and conclusion}
The present article considered the representation of dependent random variables in terms of dependence hazard measures 
and the corresponding dependence parts of the joint survival function.
Furthermore, we have examined a useful representation of correlated frailty models 
which led to the expression of higher-dimensional dependence hazard rates in terms of moments of an exponential family.
In particular, these rates are simply covariances of frailty variables for dimension 2.
It has also been shown that correlated frailty models are minimum-infinitely divisible iff its frailty vector is sum-infinitely divisible.
This equivalence yields an interesting one-to-one correspondence between the L\'evy meausure of the frailty vector
and the exponent measure of the minimum-infinitely divisible correlated frailty model.
It follows that analytical properties of such L\'evy measures might as well be examined in the context of extreme value theory.

Throughout the article we have made strong use of the copula concept 
which enabled us to use arbitrary, continuous marginal distribution functions.
Here, the special shared frailty model had been shown to have a natural connection to Archimedean copulas.
Eventually, we have utilized this concept to analyze particularly interesting semiparametric quantities of dependence.
In the final chapter an isometry between Hilbert spaces exposed the relation of this semiparametric function to tangents in the space of square-integrable functions.
It is pointed out how score functions for dependence models can be transformed in terms of hazard dependence quantities.
\section*{Acknowledgements}
\noindent This article has been developed at the Mathematical Institute of the University of D\"usseldorf, Germany,
while this had been the permanent address of all three authors.
\appendix
\section{Continuation and formulas of Example~\ref{Bsp:CFM_ZiQuadrat}(a)}
\label{sec:formulas_chi_square}
	\noindent The joint survival function of $(X_1,X_2,X_3)$ at time $\b t = (t_1,t_2,t_3)$ is given by
	\begin{align*}
		S(\b t) 
			&= \left(8 \Delta t_1 t_2 t_3 + 4 \Delta_{12} t_1 t_2 + 4 \Delta_{13} t_1 t_3 + 4 \Delta_{23} t_2 t_3 + \sum_{j=1}^3 2 \sigma_j^2 t_j + 1 \right)^{-\frac 1 2},
	\end{align*}
	where 
		$\Delta = 	\sigma_1^2 \sigma_2^2 \sigma_3^2 
				+ 2 \sigma_{12} \sigma_{13} \sigma_{23} 
				- \sigma_{12}^2 \sigma_{3}^2 - \sigma_{13}^2\sigma_{2}^2 - \sigma_{23}^2\sigma_{1}^2 $ 
	denotes the determinant of $\bs \Sigma$ 
	and 
		$\Delta_{ij}=\sigma_i^2\sigma_j^2 - \sigma_{ij}^2$ 
	is the determinant of the covariance matrix
	\begin{align*}
		\begin{pmatrix} 
				\sigma_i^2 & \sigma_{ij} 
				\\ \sigma_{ij} & \sigma_j^2 
		\end{pmatrix}
	\end{align*}
	of $Z_i$ and $Z_j$.
	The marginal survival function of $X_i$, $i \in \{1,2,3\}$, is given by
	\begin{align*}
		S_i(t_i) = S(t_{i}, \b 0) = \left( 1 + 2 \sigma_i^2 t_i \right)^{-\frac 1 2}
	\end{align*}
	so that
	\begin{align} \begin{split} \label{eq:triv_surv_fktn}
		S(\b t) & = 	\bigg\{ \frac {\Delta} {\sigma_1^2 \sigma_2^2 \sigma_3^2 } \prod_{i=1}^3 \left( S_i^{-2}(t_i) -1 \right)
		    \\ &	+ \sum_{(i,j)} \frac {\Delta_{ij}} {\sigma_i^2 \sigma_j^2} \left( S_i^{-2}(t_i) -1 \right) \left( S_j^{-2}(t_j) -1 \right)
				+ \sum_{i=1}^3  S_i^{-2}(t_i) - 1
				\bigg\}^{-\frac 1 2}
	\end{split} \end{align}
	holds; see \ref{sec:proof_formula} for a detailed derivation. Note that 
	\begin{align*}
		\frac {\Delta} {\sigma_1^2 \sigma_2^2 \sigma_3^2 } = 1 + 2 \prod_{(i,j)} \varrho_{ij} - \sum_{(i,j)} \varrho_{ij}^2 
		\quad \text{ and } \quad 
		\frac {\Delta_{ij}} {\sigma_i^2 \sigma_j^2} =  1- \varrho_{ij}^2,
	\end{align*}
	i.e.
	the survival copula $S \circ (S_1^{ -1},S_2^{ -1},S_3^{ -1})$ only depends on the bivariate parameters $\varrho_{ij}$.
	Similarly, the bivariate survival function of $X_i$ and $X_j$ 
	is given by
	\begin{align*}
		S^{\{i,j\}}(t_i,t_j) 	&= S(\b t_{\{i,j\}}, 0) = \left( 1 + 2 \sigma_i^2 t_i + 2 \sigma_j^2 t_j + 4 \Delta_{ij} t_i t_j  \right)^{-\frac 1 2}
					\\
					&= \left( S_i^{-2}(t_i) S_j^{-2}(t_j) - \varrho_{ij}^2 \left[S_i^{-2}(t_i) -1 \right]  \left[ S_j^{-2}(t_j) -1 \right] \right)^{-\frac 1 2},
	\end{align*}
	see also \citet[Section 5.6]{Wienke11}. 
	The dependence part of order 2 is given by 
	\begin{align*}
		S_{\{i,j\}}(t_i,t_j) 	= \frac {S^{\{i,j\}}(t_i,t_j)} {S_i(t_i) S_j(t_j)} 
					= \left(1 - \varrho_{ij}^2 \left[ 1 - S_i^2(t_i) \right]  \left[ 1- S_j^2(t_j) \right] \right)^{-\frac 1 2}
	\end{align*}
	with hazard density
	\begin{align*}
		\lambda_{\{i,j\}}(t_i,t_j) = 2 \sigma_{\{i,j\}}^2 \left[S^{\{i,j\}}(t_i,t_j)\right]^4.
	\end{align*}
	The dependence part of order 3 is given by $S_{\{1,2,3\}}(\b t)$
	\begin{align*}
		= \left\{ \frac 	
		{ \prod_{(i,j)} \left( -1 + \varrho_{ij}^2 [1-S_i^2(t_i)] [1-S_j^2(t_j)] \right) }
							{ 	-1 
									+ \sum_{(i,j)} \varrho_{ij}^2 [1 - S_i^{2}(t_i)] [1 - S_j^{2}(t_j)] 
									- 2 \prod_{(i,j)} \varrho_{ij}^2 \prod_{i=1}^3 (1-S_i^2(t_i)) } 
						\right\}^{\frac 1 2}.
	\end{align*}
	Note that $\varrho_{12} = \varrho_{13} = 0$ already implies $S_{\{1,2,3\}} \equiv 1$ as well as $S = S_1 S^{\{2,3\}}$, even if $\varrho_{23} \neq 0$.

\section{Proof of Formula~\eqref{eq:triv_surv_fktn}}
\label{sec:proof_formula}
\noindent Let $\Sigma^{-1}$ denote the inverse of the covariance matrix $\Sigma$ and let
\begin{align*}
	(g_{ij})_{i,j=1,2,3} := \Delta \Sigma^{-1} =
			\begin{pmatrix} 
				\sigma_2^2\sigma_3^2 - \sigma_{23}^2 		& \sigma_{13}\sigma_{23}-\sigma_{12}\sigma_3^2 		& \sigma_{12} \sigma_{23}-\sigma_{13}\sigma_2^2
				\\ 
				\sigma_{13}\sigma_{23}-\sigma_{12}\sigma_3^2 	& \sigma_1^2\sigma_2^2 - \sigma_{12}^2  			& \sigma_{13}\sigma_{12} - \sigma_1^2 \sigma_{23}
				\\ 
				\sigma_{12} \sigma_{23}-\sigma{13}\sigma_2^2	&\sigma_{13}\sigma_{12} - \sigma_1^2 \sigma_{23} 	& \sigma_1^2\sigma_2^2 - \sigma_{12}^2 
			\end{pmatrix}
\end{align*}
define the matrix $(g_{ij})_{i,j \in \{1,2,3\}}$.
The joint probability density function of $(Z_1,Z_2,Z_3)$ is given by
\begin{align*}
	f(\b z) = \frac 1 {\sqrt{(2\pi)^3 \Delta}} \exp\bigg(-\frac { \Sigma^{-1} } 2  \big\{ 	g_{11}z_1^2 + g_{22}z_2^2 + g_{33}z_3^2 
		 \\ + 2 g_{12} z_1 z_2 + 2 g_{13} z_1 z_3 + 2 g_{23} z_2 z_3 \big\}\bigg).
\end{align*}
By Lemma~\ref{lemma:surv_eq_lapl} the survival function of the correlated frailty model $\b X = \big(\frac{Y_i}{Z_i^2}\big)_{i \in \{1,2,3\}}$ is given by
\begin{align*}
	S(\b t) = \int \int \int \exp\left( - \sum_{i=1}^3 t_i z_i^2 \right) f(\b z) \intd z_1 \intd z_2 \intd z_3.
\end{align*}
Consider
\begin{align*}
	E :&=
	\int \int \int 	\exp\left( - \sum_{i=1}^3 t_i z_i^2 \right)  
	\\
	& \, \times
			\exp	\left(-\frac { \Sigma^{-1} } 2 	\left\{ g_{11}z_1^2 + g_{22}z_2^2 + g_{33}z_3^2 
										+ 2 g_{12} z_1 z_2 + 2 g_{13} z_1 z_3 + 2 g_{23} z_2 z_3 
									\right\}
				\right) 
	\\
	& \quad \intd z_1 \intd z_2 \intd z_3 
	\\
	&= \int \int \exp\left\{ -t_2z_2^2 - t_3 z_3^2 - \frac 1 {2\Delta} g_{22} z_2^2 - \frac 1 {2\Delta} g_{33} z_3^2 - \frac 1 \Delta g_{23} z_2 z_3 \right\}
	\\
	& \, \times \int \exp\left\{ -t_1 z_1^2 - \frac 1 {2 \Delta} g_{11} z_1^2 - \frac 1 \Delta \left[ g_{12} z_2 + g_{13} z_3\right] z_1 \right\} \intd z_1 \intd z_2 \intd z_3.
\end{align*}
Defining
\begin{align*}
	a_1 := t_1 + \frac {g_{11}} {2\Delta},\  b_1 := \frac 1 {2 \Delta} \left[ g_{12} z_2 + g_{13} z_3 \right],\ y_1 := 2 \Delta t_1 + g_{22}
\end{align*}
we obtain
\begin{align*}
		\int &\exp\left\{ -t_1 z_1^2 - \frac 1 {2 \Delta} g_{11} z_1^2 - \frac 1 \Delta \left[ g_{12} z_2 + g_{13} z_3 \right] z_1 \right\} \intd z_1
		\\
		&= \int \exp\left\{ -a_1 z_1^2 - 2 b_1 z_1 \right\} \intd z_1
		\\
		&= \sqrt{\frac \pi {a_1} } \exp\left( \frac {b_1^2} {a_1} \right)
		\\
		&= \sqrt{ \frac {2 \pi \Delta} {y_1} }  \exp\left( \frac 1 {2 \Delta y_1} \left[ g_{12}^2 z_2^2 + 2 g_{12} g_{13} z_2 z_3 + g_{13}^2 z_3^2 \right]^2 \right)
\end{align*}
and therefore
\begin{align*}
	E	&= \sqrt{ \frac {2 \pi \Delta} {y_1} } \int \int \exp\bigg\{ - \left[ t_2 + \frac {g_{22}} {2 \Delta} - \frac {g_{12}^2} {2 \Delta y_1} \right] z_2^2
		\\
		& \qquad - 2 \left[ \left( \frac 1 {2\Delta} g_{23} - \frac {g_{12} g_{13}} {2\Delta y_1} \right) z_3 \right] z_2 \bigg\} \intd z_2
		\\
		&\qquad \quad \times 
			\exp\left\{- \left[t_3 + \frac {g_{33}} {2\Delta} - \frac {g_{13}^2} {2\Delta y_1} \right] z_3^2 \right\} \intd z_3.  
\end{align*}
With
\begin{align*}
	y_2 := 2 \Delta t_2 + g_{22},\ a_2 := \frac { 2 t_2 \Delta y_1 + g_{22} y_1 - g_{12}^2 }  {2 \Delta y_1} = \frac {y_1 y_2 - g_{12}^2} {2 \Delta y_1}
\end{align*}
and
\begin{align*}
	b_2 := \frac {z_3} {2\Delta} \left( g_{23} - \frac {g_{12} g_{13}} {y_1} \right),\ y_3 := 2 \Delta t_3 + g_{33}
\end{align*}
we obtain
\begin{align*}
	\int &\exp\left\{ - \left[ t_2 + \frac {g_{22}} {2 \Delta} - \frac {g_{12}^2} {2 \Delta y_1} \right] z_2^2
			- 2 \left[ \left( \frac 1 {2\Delta} g_{23} - \frac {g_{12} g_{13}} {2\Delta y_1} \right) z_3 \right] z_2 \right\} \intd z_2
		\\
		&= 	\int \exp\left\{ - a_2 z_2^2 - 2 b_2 z_2 \right\} \intd z_2
		\\
		&=	\sqrt{\frac \pi {a_2} } \exp\left( \frac {b_2^2} {a_2} \right)
		\\
		&=	\sqrt{ \frac { 2 \pi \Delta y_1 }  { y_1 y_2 - g_{12}^2 } }  
			\exp\left\{ \frac { y_1 }  { y_1 y_2 - g_{12}^2 }\frac {z_3^2}  {2\Delta}  \left( g_{23}^2 - 2 \frac {g_{12}g_{13}g_{23}} {y_1} + \frac {g_{12}^2g_{13}^2} {y_1^2} \right) \right\}
\end{align*}
and thus
\begin{align*}
	S(\b t) &= \sqrt{\frac {\Delta} {2\pi \left(y_1y_2 - g_{12}^2\right) } }  
		\\
		&\qquad\times
			\int \exp	\bigg\{ - z_3^2 	\bigg[ \frac{y_1 y_3 - g_{13}^2} {2\Delta y_1} - \frac {y_1} { 2\Delta \left(y_1 y_2 - g_{12}^2\right)}  
									\\
									& \qquad \quad \times \left( g_{23}^2 - 2 \frac {g_{12}g_{13}g_{23}} {y_1} + \frac {g_{12}^2g_{13}^2} {y_1^2} \right)
								\bigg]
					\bigg\}
			\intd z_3
		\\
		&= \sqrt{\frac {\Delta} {2\pi \left(y_1y_2 - g_{12}^2\right) } }  
		\\
		&\qquad\times
			\int \exp	\left\{ - z_3^2 
						\left[ \frac {y_1 y_2 y_3 - g_{23}^2 y_1 - g_{13}^2 y_2 - g_{12}^2 y_3 + 2 g_{12} g_{13} g_{23} }  {2 \Delta \left(y_1 y_2 - g_{12}^2\right)} \right]
					\right\}
			\intd z_3
		\\
		&= \Delta \left(y_1 y_2 y_3 - g_{23}^2 y_1 - g_{13}^2 y_2 - g_{12}^2 y_3 + 2 g_{12} g_{13} g_{23}\right)^{-\frac 1 2}
		\\
		&= 	\big( 8 \Delta t_1 t_2 t_3 
				+ 4 g_{33} t_1 t_2  + 4 g_{22} t_1 t_3 + 4 g_{11} t_2 t_3 
				\\
				& \qquad + 2 \sigma_1^2 t_1 + 2\sigma_2^2 t_2 + 2\sigma_3^2 t_3
				+1
			\big)^{-\frac 1 2}.
\end{align*}
With $S_i^{-1}(u_i) = \frac {u_i^{-2} -1} {2\sigma_i^2}$ this yields \eqref{eq:triv_surv_fktn}.

\section{Further Proofs}
\label{sec:proofs}

\begin{proof}[Proof of Example~\ref{Bsp:doppelt}]
  (a) is obvious since $(Z_1, \dots, Z_d)$ is uniquely determined by its two-dimensional marginals.\\
  (b)	  By Lemma~\ref{lemma:surv_eq_lapl}~(a) the vector $(X_1,\dots,X_d)$ 
	  has independent components iff $(W_1,\dots,W_d)$ does.
	  Thus, the proof for~\eqref{item:log-normal_frailty} is trivial since the exponential function is a bijective mapping.
	  To prove (b) for~\eqref{item:chi2_frailty} we introduce $a = Cov(Z_i, Z_j) Var(Z_i)^{-1}$ and we show
	  that $Z_i^2$ is independent of $Z_j^2$ iff $Z_i$ is independent of $Z_j$ for arbitrary indices $i \neq j$.
	  Without loss of generality let $E(Z_i) = E(Z_j) = 0$ and $Var(Z_i) = Var(Z_j) = 1$.
	  It is well known that, for a multivariate normally distributed vector $(Z_i,Z_j)$, 
	  the random variables $Z_i$ and $Z_j - a Z_i$ are independent.
	  Now consider $Cov(Z_i^2, Z_j^2) = E((Z_i^2-1)(Z_j^2-1))$ and repeatedly replace $Z_j$ by $(Z_j - a Z_i) + a Z_i$
	  so that this covariance is expressed through expectations only depending on $Z_i$ and $Z_j - a Z_i$.
	  The independence of these two random variables yields
	  \begin{align*}
	   	Cov(Z_i^2,Z_j^2) & = E( (Z_i^2 - 1)(Z_j^2 - 1) ) \\
	   	& = E( (Z_i^2 - 1)((Z_j - a Z_i)^2 + a^2 Z_i^2 + 2 a Z_i (Z_j - a Z_i) - 1) ) \\
	   	& = 0 + a^2 E( (Z_i^2 - 1) Z_i^2) + 0 - 0
	   	 = a^2 (E(Z_i^4) - 1).
	  \end{align*}
	  Hence $Cov(Z_i^2,Z_j^2) = 0$ iff $a = 0$ iff $Cov(Z_i,Z_j) = 0$.
\end{proof}

\begin{proof}[Remark] \label{rem:cfm_sum_id_iff_min_id}
 Let $\b X$ be a correlated frailty model with frailty random vector $\b W$.
 Then the survival function $S$ of $\b X$ is minimum-infinitely disvisible
 iff $\mathfrak{L}(\b W)$ is sum-infinitely divisible.
 \\ \emph{Proof.} Having the sum-infinite divisibility of $\mathfrak{L}(\b W)$,
  the minimum-infinite divisibility of $S$ is obvious:
  this is a simple generalization of Lemma~\ref{lem:min_frailty}(b).
  
  On the other hand, if $S = (S_n)^n$ is minimum-infinitely divisible,
  then it is easy to see via induction that for all $n_1, \dots, n_d \in \mathbb N_0$ and $\b t \in (0,\infty)^d$
  \begin{align*}
   (-1)^{\sum_{i=1}^d n_i} \bigg( \prod_{i \in I} \frac{\partial^{n_i}}{\partial t_i^{n_i}} \bigg) S_n(\b t) \geq 0.
  \end{align*}
  An application of a multivariate version of the Bernstein-Widder theorem (Theorem~1.3.1 of~\cite{Zocher05})
  now shows that $S_n$ is the Laplace transform of a probability measure on $\mathbb{R}^d$.
  Hence, $S_n$ is the survival function of a $d$-dimensional correlated frailty model.
  Finally, another application of Lemma~\ref{lem:min_frailty}(b) 
  and the fact that each distribution determines a unique Laplace transform
  show the sum-infinite divisibility of $\mathfrak{L}(\b W)$.
\end{proof}

\begin{proof}[Proof of Lemma~\ref{lem:lambda_id_w}]
	  Inserting the definition of $\lambda_I$ and the L\'evy-Khintchine formula yields
	\begin{align*}
		 \lambda_I(\b t_I) = & \frac{\tn d}{\tn d \b t_I} \Lambda_I(\b t_I) 
		  = (-1)^{|I|} \frac{\tn d}{\tn d \b t_I} \log \psi_{Q_\b 0}(\b t_I, \b 0) \\
		  = & (-1)^{|I|} \frac{\tn d}{\tn d \b t_I} \Big( -<\b t_I,\b b_I> + \int_{\Rdpo0} \left( e^{-<\b t_I,\b x_I>} - 1 \right) \tn d \eta_{\b 0}(\b x) \Big). 
	\end{align*}
	Hence, the proof for $I = \{i\}$ is covered by the following considerations for a general index set $I \neq \emptyset$.
	Note that, for fixed $\b t_I \in (0,\infty)^{|I|}$, the absoulte value of the function
	\begin{align*}
		\b x_I \mapsto \frac{\tn d}{\tn d \b t_I} \left( e^{-<\b t_I,\b x_I>} - 1 \right) 
		= (-1)^{|I|} \prod_{i \in I} x_i e^{-<\b t_I,\b x_I>}
	\end{align*}
	is bounded by an $\eta_{\b 0}$-integrable function which is independent of $\b t_I$:
	\begin{align*}
		 \Big | \prod_{i \in I} x_i e^{-<\b t_I,\b x_I>}  \Big | \leq \| \b x_I \| 1\{\| \b x_I \| \leq 1\} + M 1\{\| \b x_I \| > 1\}
		  \leq M \min(1,\|\b x_I\|), 
	\end{align*}
	for a constant $M \geq 1$, and this is integrable with respect to $\eta_\b 0$. 
	It follows that differentiation and integration are allowed to change places 
	and eventually the definition of $\eta_{(\b t_I, \b 0)}$ concludes the proof via induction over $d$.
\end{proof}

\begin{proof}[Proof of Lemma~\ref{lem:Lambda_id_w}]
 The proof for $|I|=1$ is a simpler version of the following proof; thus, it is left to the reader.
 Applying Lemma~\ref{lem:lambda_id_w} and Fubini's theorem we get
 \begin{align*}
  \Lambda_I(\b t_I) = & \int_{\b 0}^{\b t_I} \lambda_I(\b s_I) \tn d \b s_I 
    =  \int_{\b 0}^{\b t_I} \int_{\Rdpo0} \prod_{i \in I} x_i \exp(-x_i s_i) \tn d \eta_\b 0(\b x) \tn d \b s_I \\
    = & \int_{\Rdpo0} \prod_{i \in I} \Big( \int_0^{t_i} x_i \exp(-x_i s_i) \tn d s_i \Big) \tn d \eta_\b 0(\b x) \\
    = & \int_{\Rdpo0} \prod_{i \in I} ( 1 - \exp(-x_i t_i)) \tn d \eta_\b 0(\b x) \\
    = & \int_{\Rdpo0} \sum_{J \subset I}  (-1)^{|J|} \prod_{j \in J} \exp(-x_j t_j) \tn d \eta_\b 0(\b x) \\
    = & \int_{\Rdpo0} \sum_{J \subset I}  (-1)^{|J|} \exp(- < \b x_J, \b t_J > ) \tn d \eta_\b 0(\b x).
 \end{align*}
 The linearity of the integral and the definition of $\eta_{(\b t_J,\b 0)}$ finish this proof.
\end{proof}

\section{Further Graphical Illustrations}
\label{sec:graphics}
\begin{figure}[H]
	\centering
	\mbox{ 
		\begin{subfigure}{0.4\textwidth}
			\includegraphics[width=\textwidth]{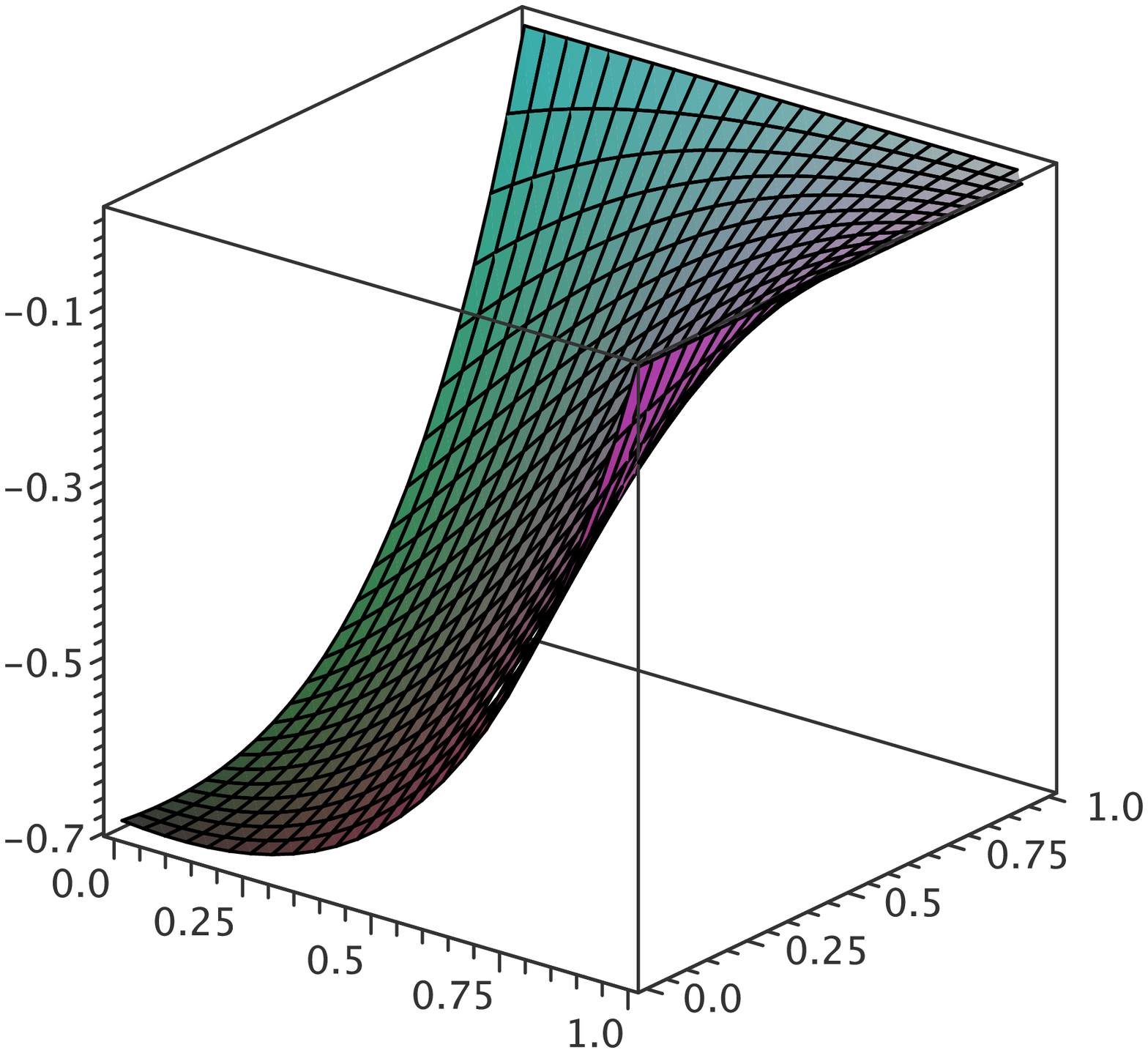}
			\caption{$\theta=-2$}
		\end{subfigure} 
		\begin{subfigure}{0.4\textwidth}
			\includegraphics[width=\textwidth]{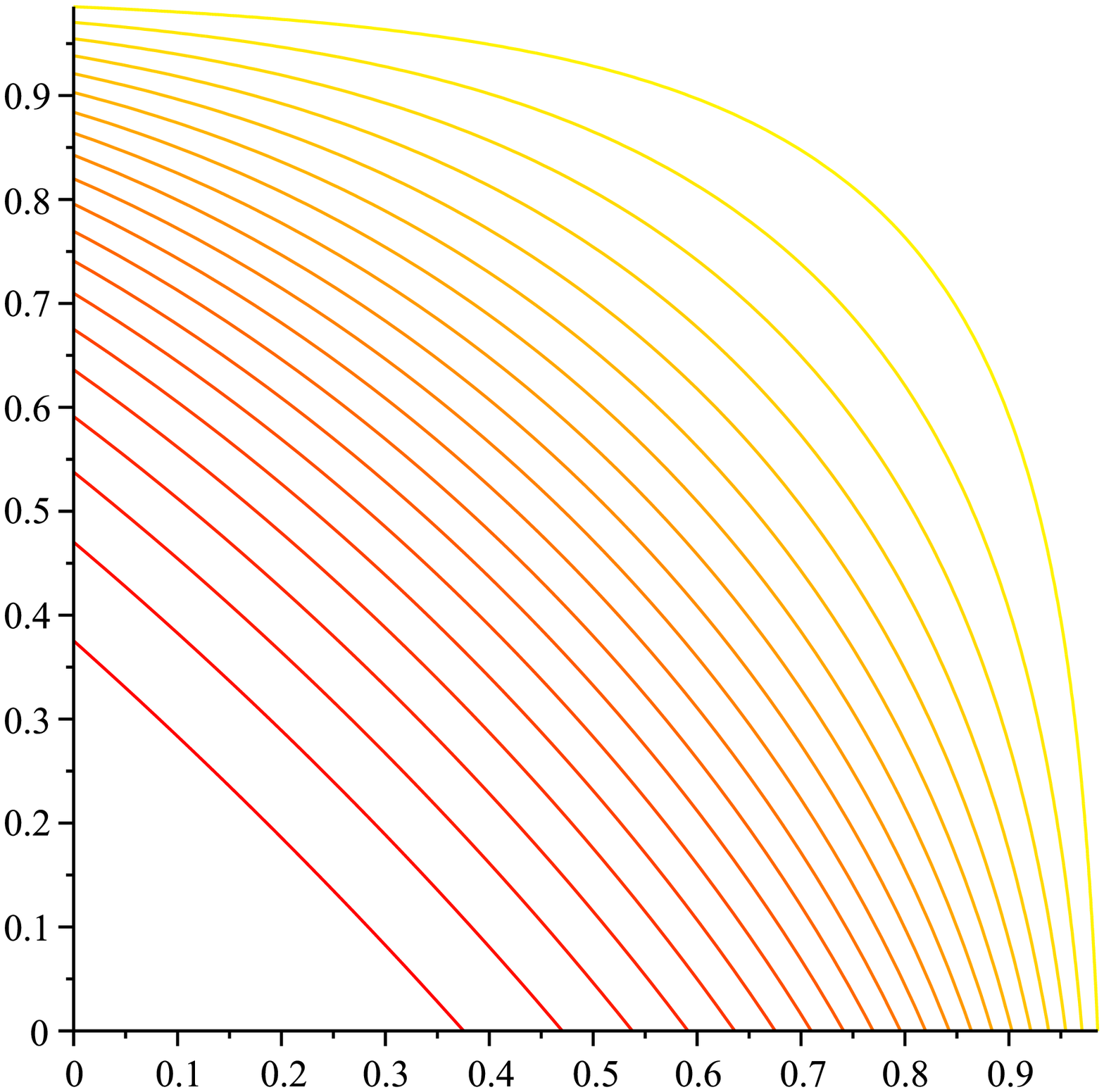}
			\caption{$\theta=-2$}
		\end{subfigure}
	} \\
	\mbox{ 
		\begin{subfigure}{0.4\textwidth}
			\includegraphics[width=\textwidth]{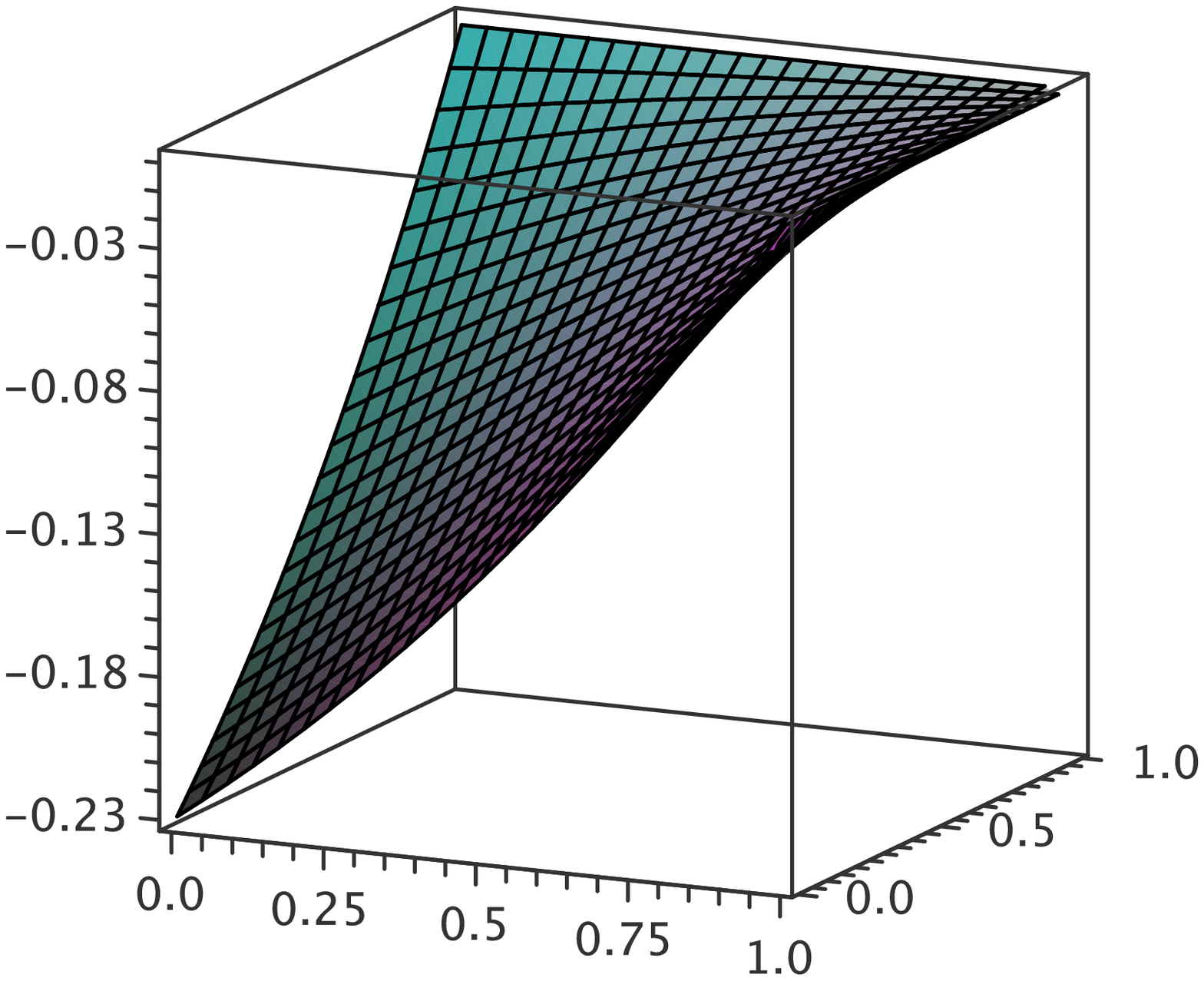}
			\caption{$\theta=-0.5$}
		\end{subfigure} 
		\begin{subfigure}{0.4\textwidth}
			\includegraphics[width=\textwidth]{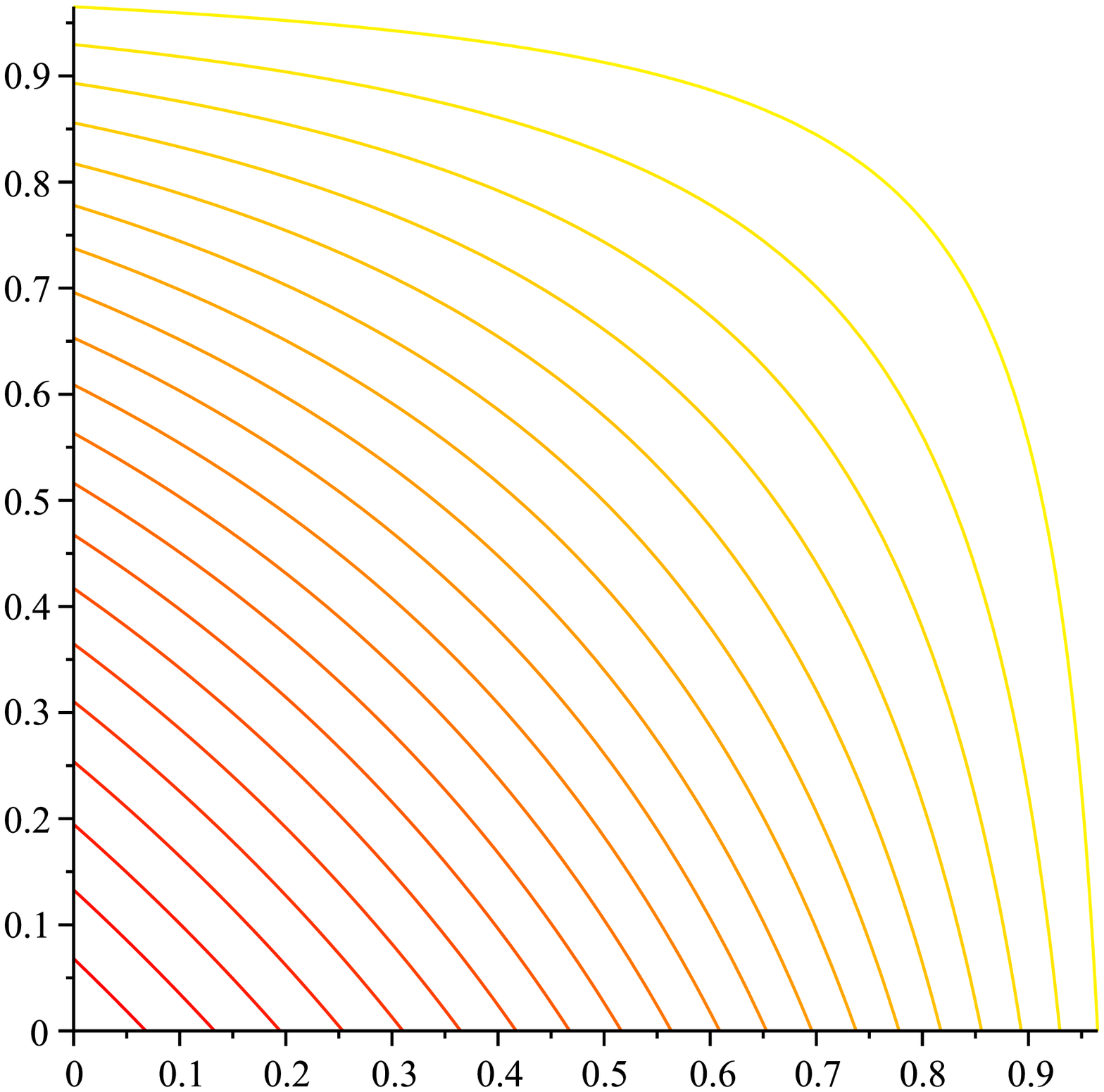}
			\caption{$\theta=-0.5$}
		\end{subfigure}
	}
	\caption{Plots and contour curves for the dependence hazard rate derivative $\gamma_{0,\{i,j\}}$ of a model with Frank copula as survival colula over the range $[0,1]^2$, with different values of $\theta$.}
\label{fig:gamma_0_fuer_frank_2}
\end{figure}
\begin{figure}[H]
	\centering
	\mbox{ 
		\begin{subfigure}{0.4\textwidth}
			\includegraphics[width=\textwidth]{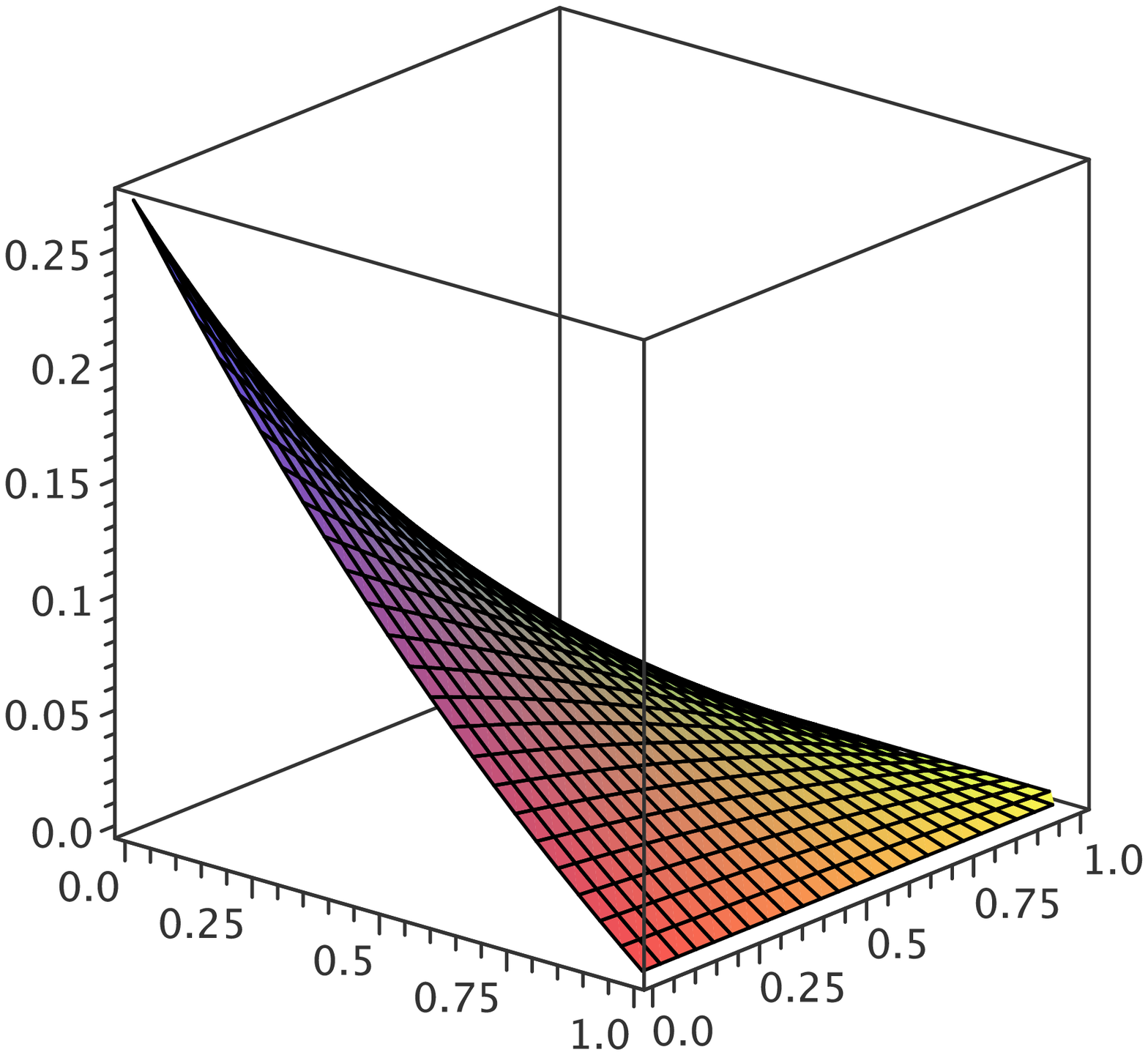}
			\caption{$\theta=0.5$}
		\end{subfigure} 
		\begin{subfigure}{0.4\textwidth}
			\includegraphics[width=\textwidth]{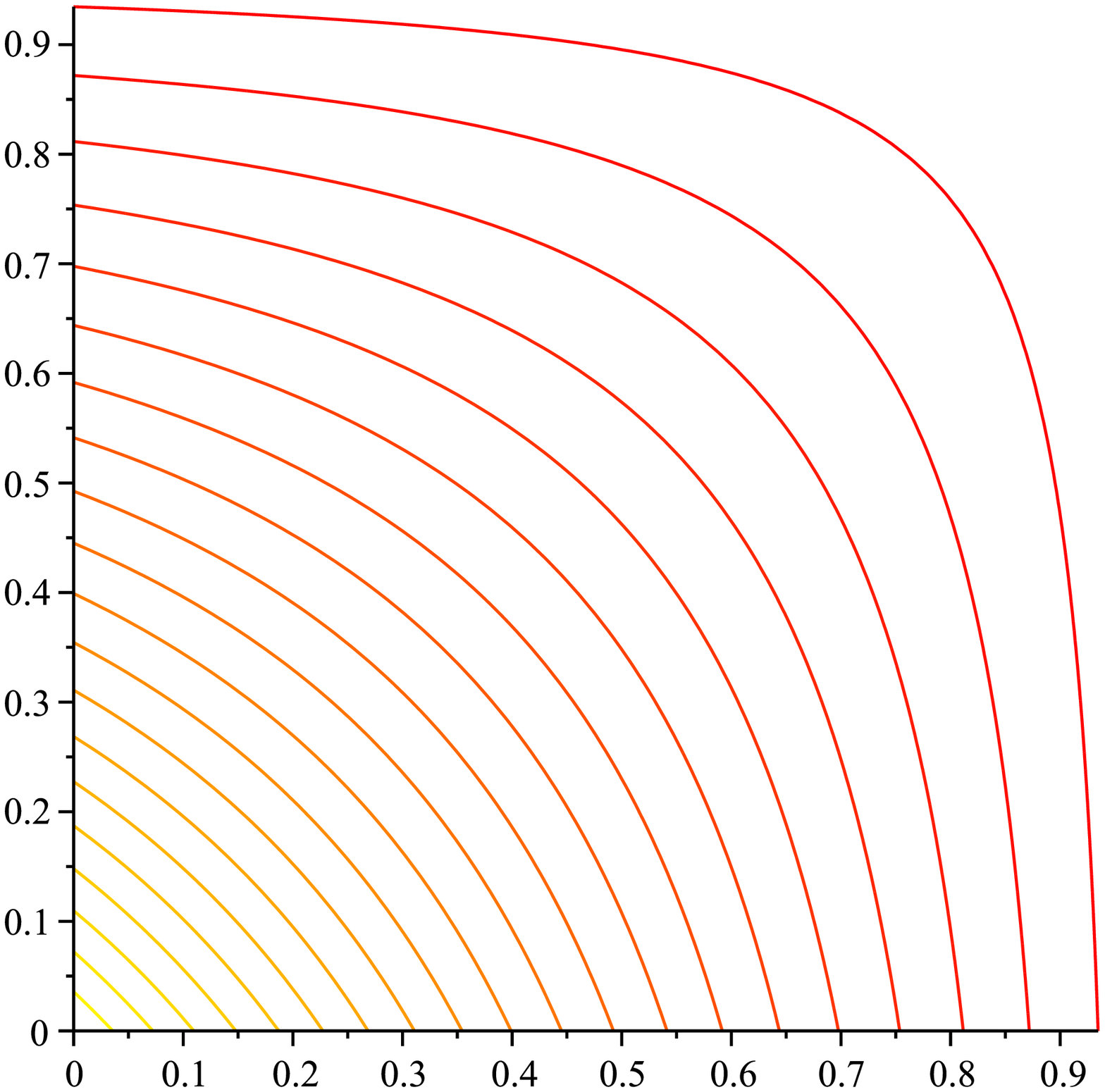}
			\caption{$\theta=0.5$}
		\end{subfigure}
	} \\
	\mbox{ 
		\begin{subfigure}{0.4\textwidth}
			\includegraphics[width=\textwidth]{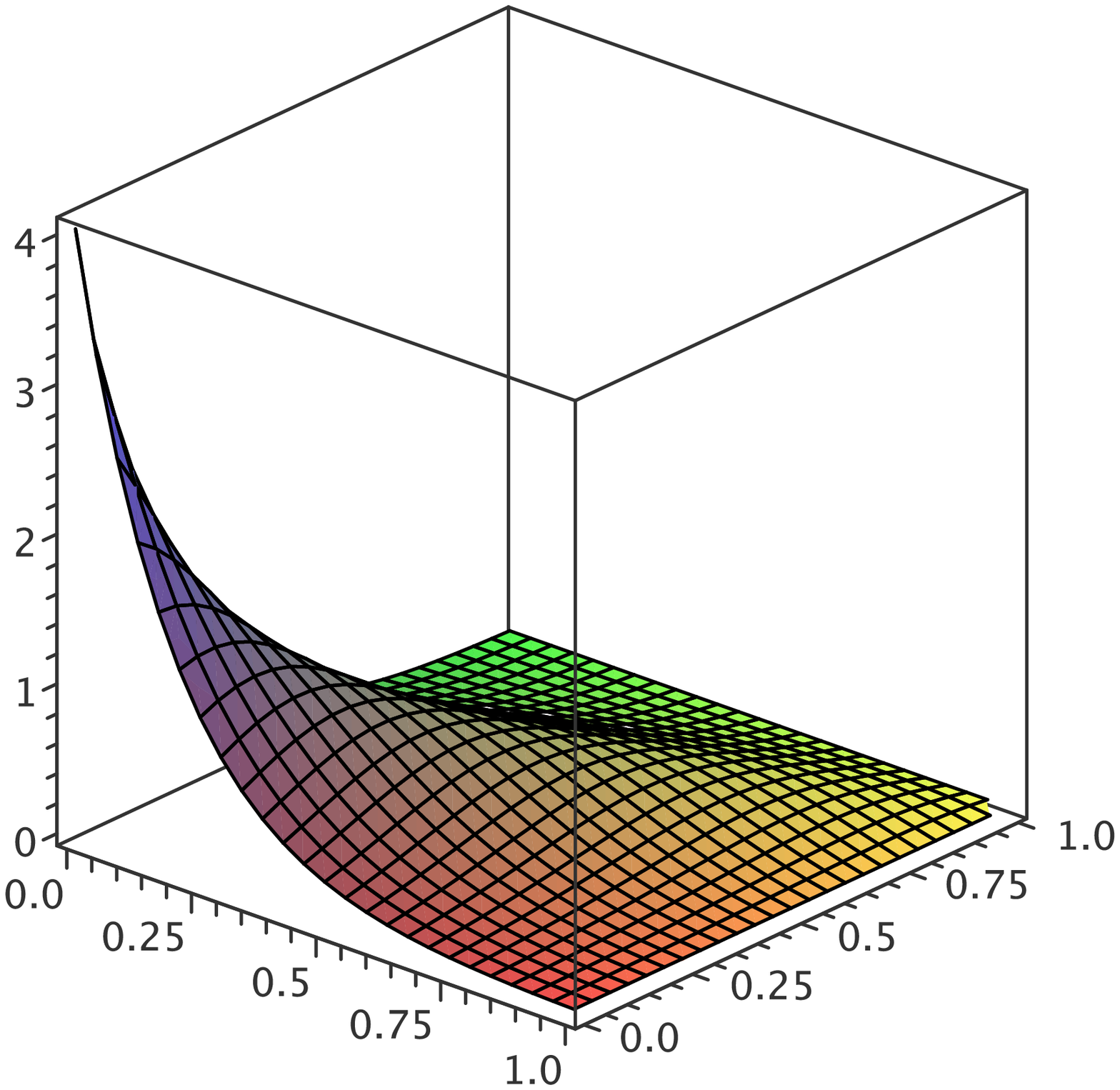}
			\caption{$\theta=5$}
		\end{subfigure} 
		\begin{subfigure}{0.4\textwidth}
			\includegraphics[width=\textwidth]{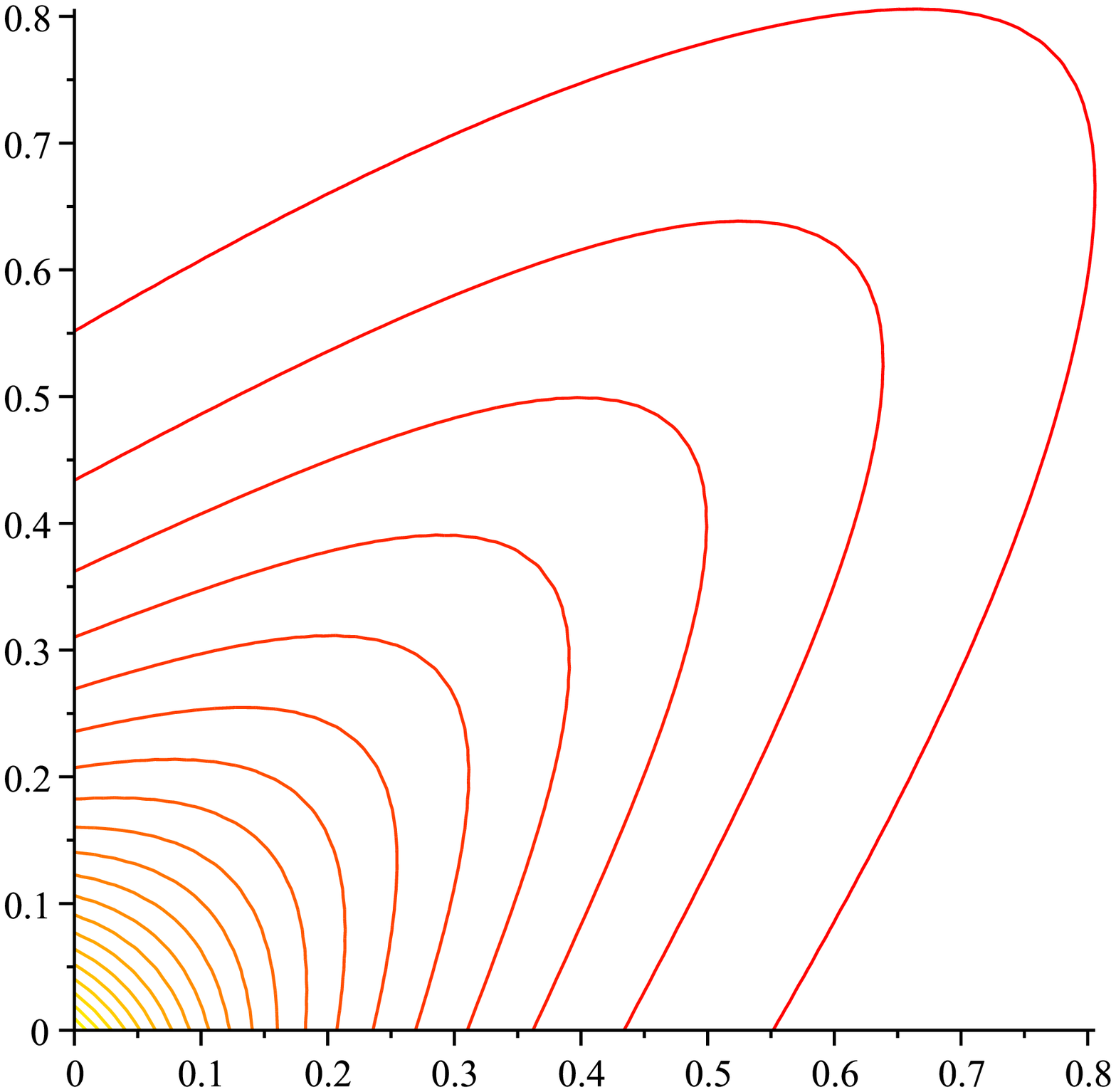}
			\caption{$\theta=5$}
		\end{subfigure}
	} \\
		\mbox{ 
		\begin{subfigure}{0.4\textwidth}
			\includegraphics[width=\textwidth]{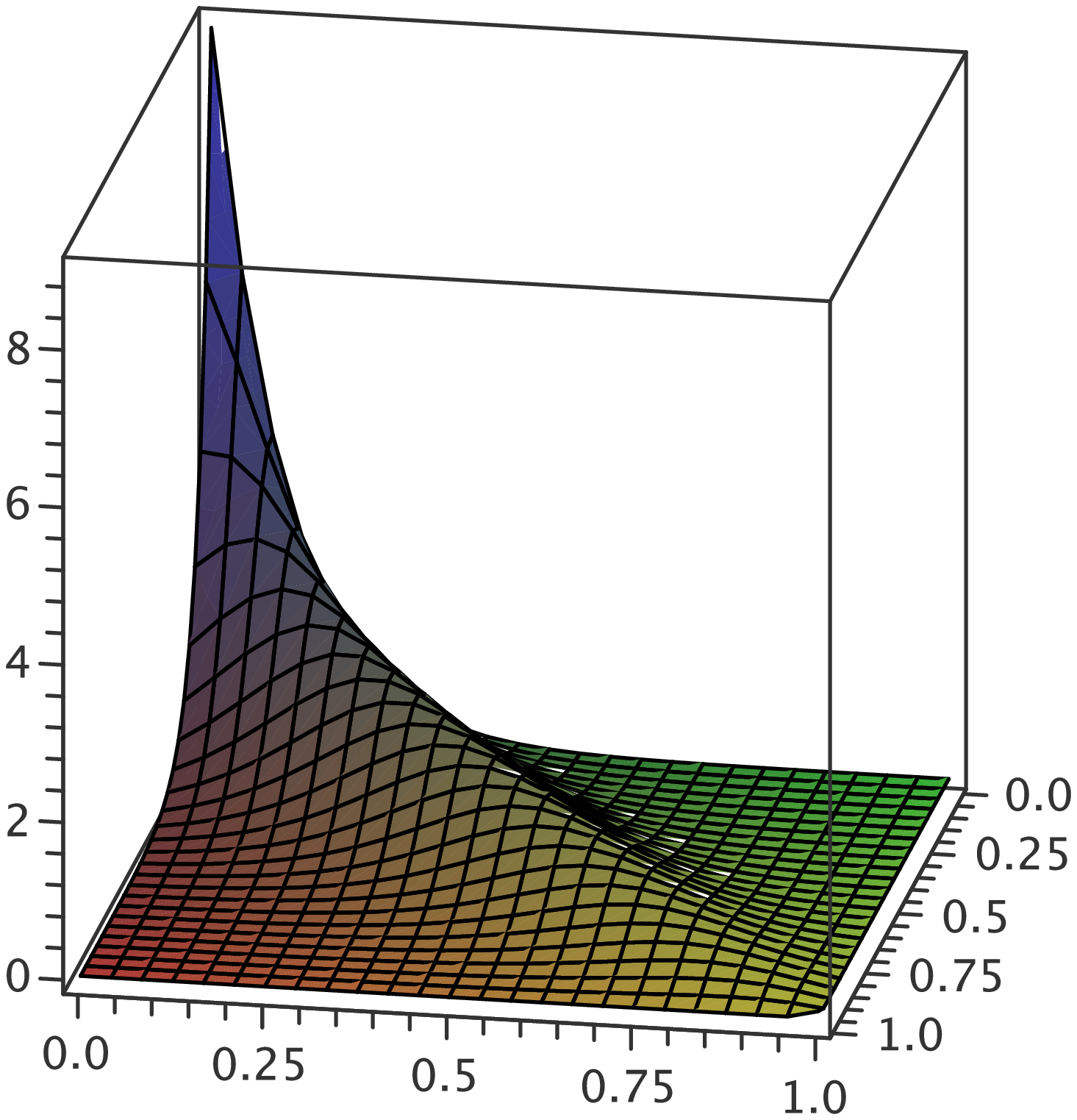}
			\caption{$\theta=10$}
		\end{subfigure} 
		\begin{subfigure}{0.4\textwidth}
			\includegraphics[width=\textwidth]{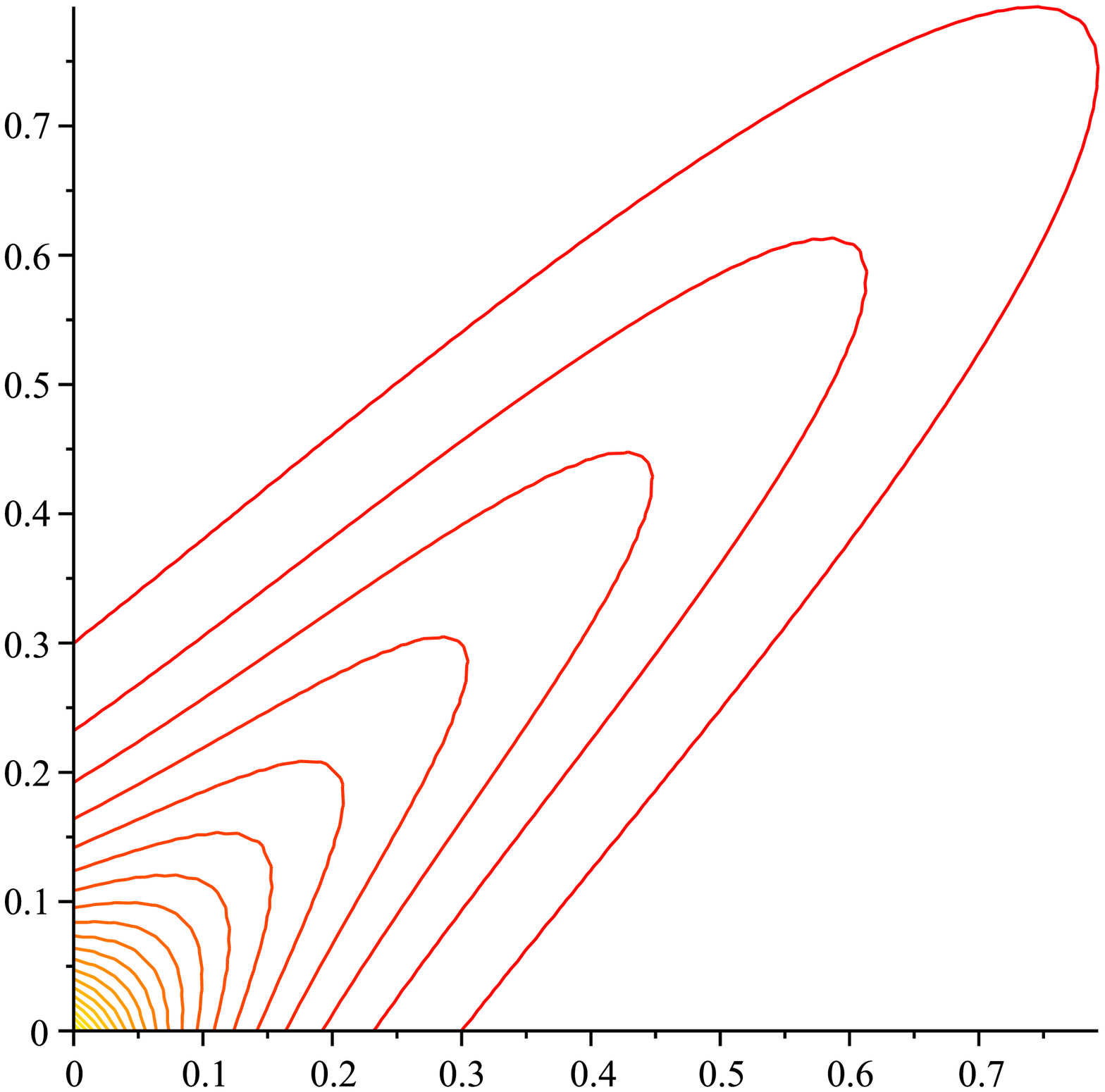}
			\caption{$\theta=10$}
		\end{subfigure}
	}
	\caption{Plots and contour curves for the dependence hazard rate derivative $\gamma_{0,\{i,j\}}$ of a model with Frank copula as survival colula over the range $[0,1]^2$, with different values of $\theta$.}
\label{fig:gamma_0_fuer_frank_3}
\end{figure}
\begin{figure}[H]
	\centering
	\mbox{ 
		\begin{subfigure}{0.4\textwidth}
			\includegraphics[width=\textwidth]{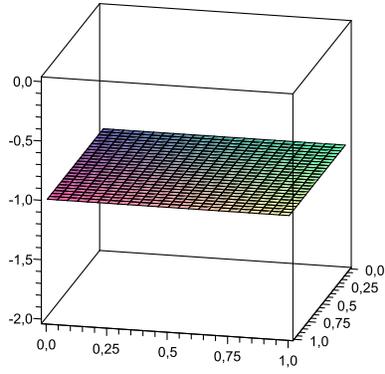}
		\end{subfigure}
	}
	\caption{Plot for the dependence hazard rate derivative $\gamma_{0,\{1,2\}} \equiv -1$ in the proportional hazard rate dependence model~\eqref{eq:415} 
	  for $\beta = 1$ over the range $[0,1]^2$. }
\label{fig:gamma_0_fuer_prop}
\end{figure}
\begin{figure}[H]
	\centering
	\mbox{ 
		\begin{subfigure}{0.4\textwidth}
			\includegraphics[width=\textwidth]{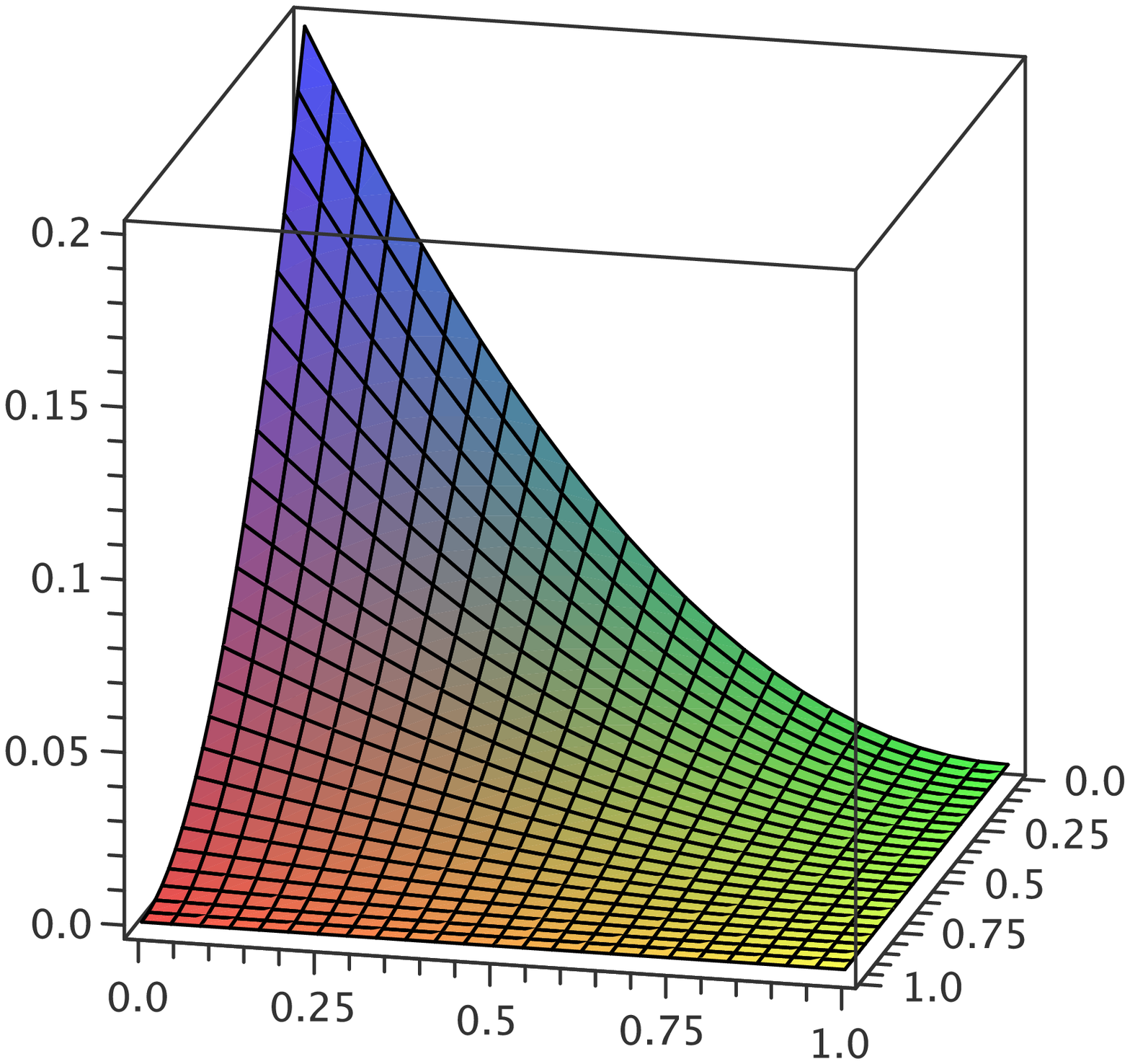}
			\caption{$\varrho_{ij}^2=0.1$}
			\label{fig_gamma:subfig1}
		\end{subfigure} 
		\begin{subfigure}{0.4\textwidth}
			\includegraphics[width=\textwidth]{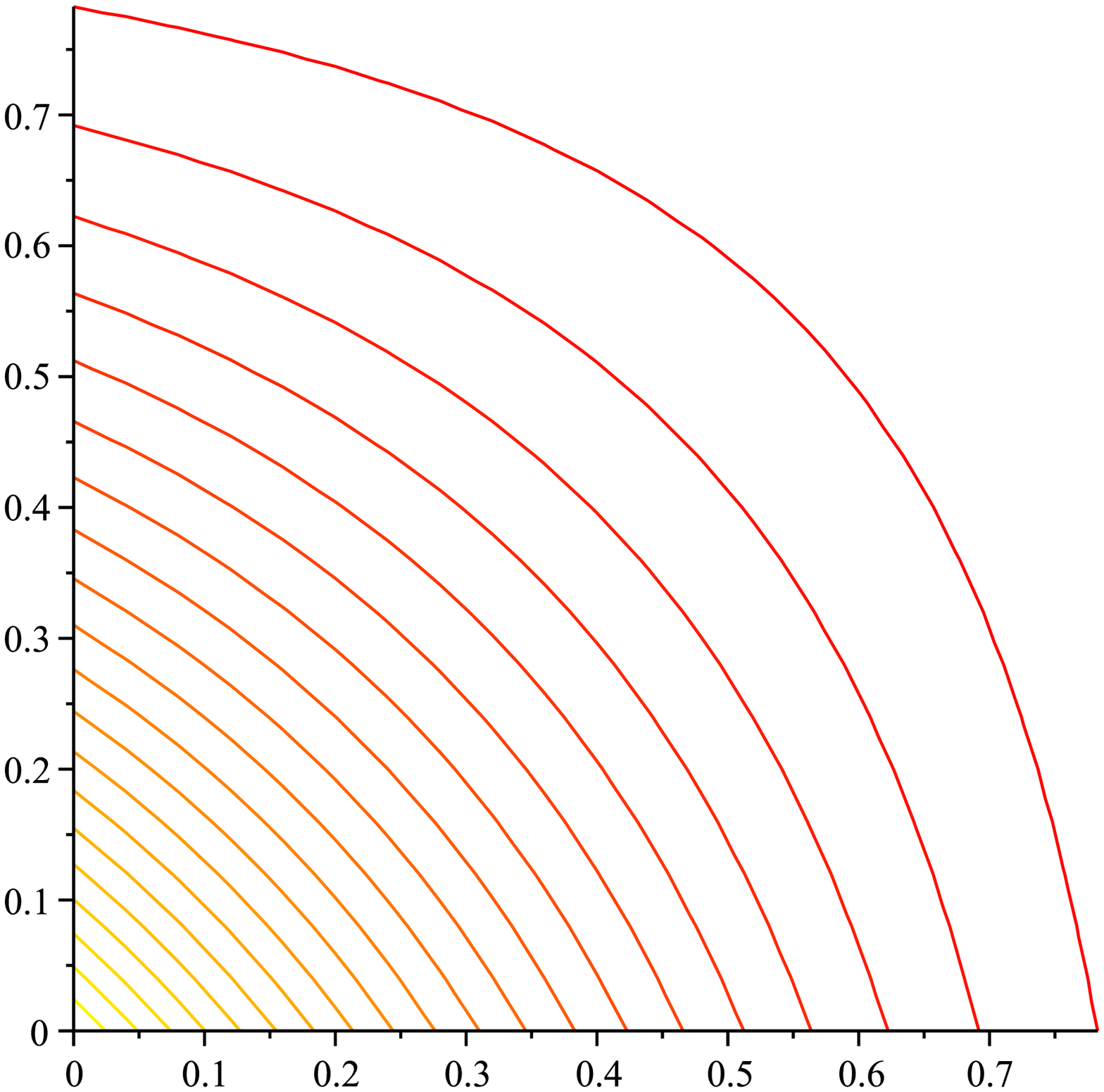}
			\caption{$\varrho_{ij}^2=0.1$}
			\label{fig_gamma:subfig2}
		\end{subfigure}
	} \\
	\mbox{ 
		\begin{subfigure}{0.4\textwidth}
			\includegraphics[width=\textwidth]{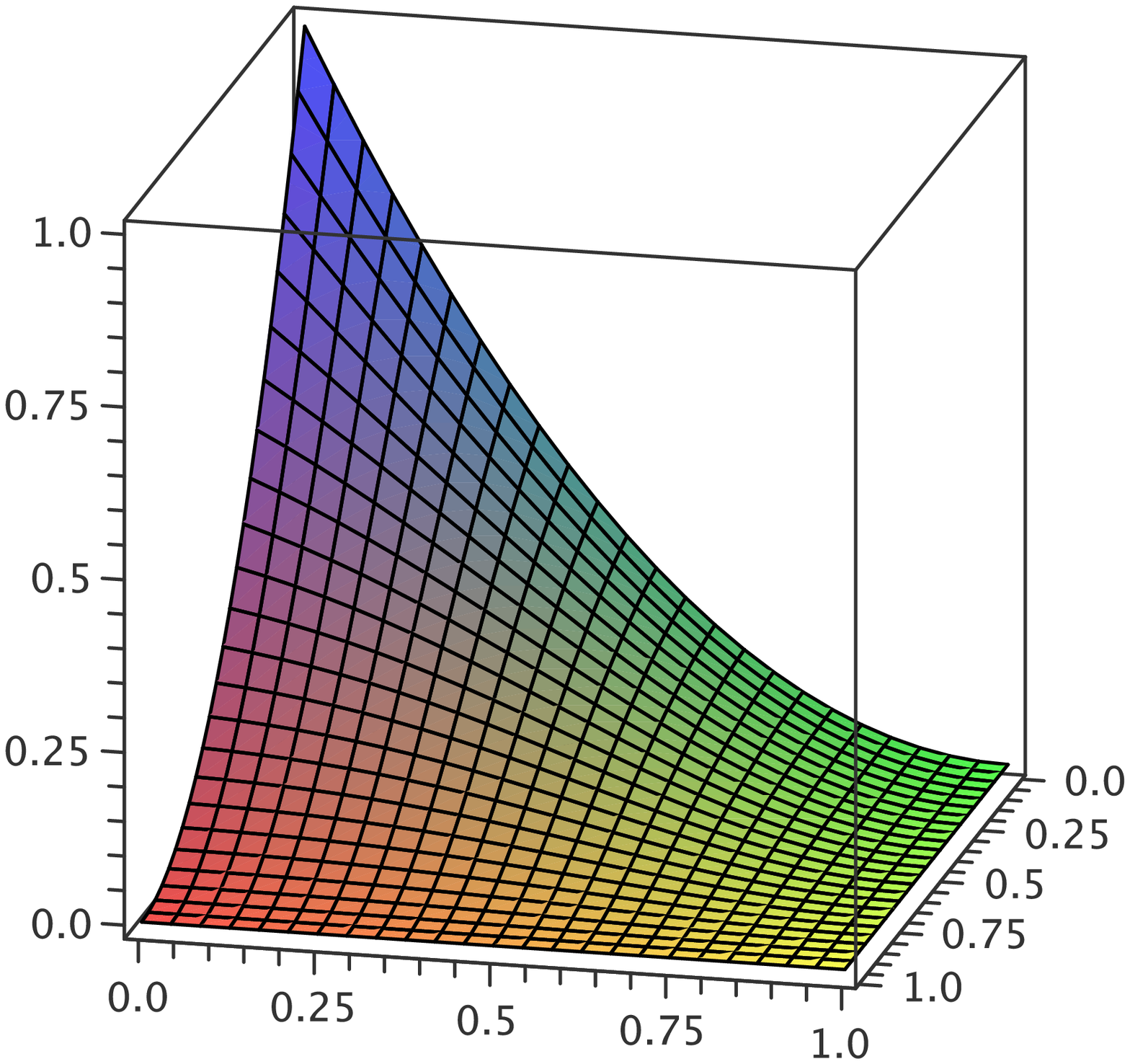}
			\caption{$\varrho_{ij}^2=0.5$}
			\label{fig_gamma:subfig3}
		\end{subfigure} 
		\begin{subfigure}{0.4\textwidth}
			\includegraphics[width=\textwidth]{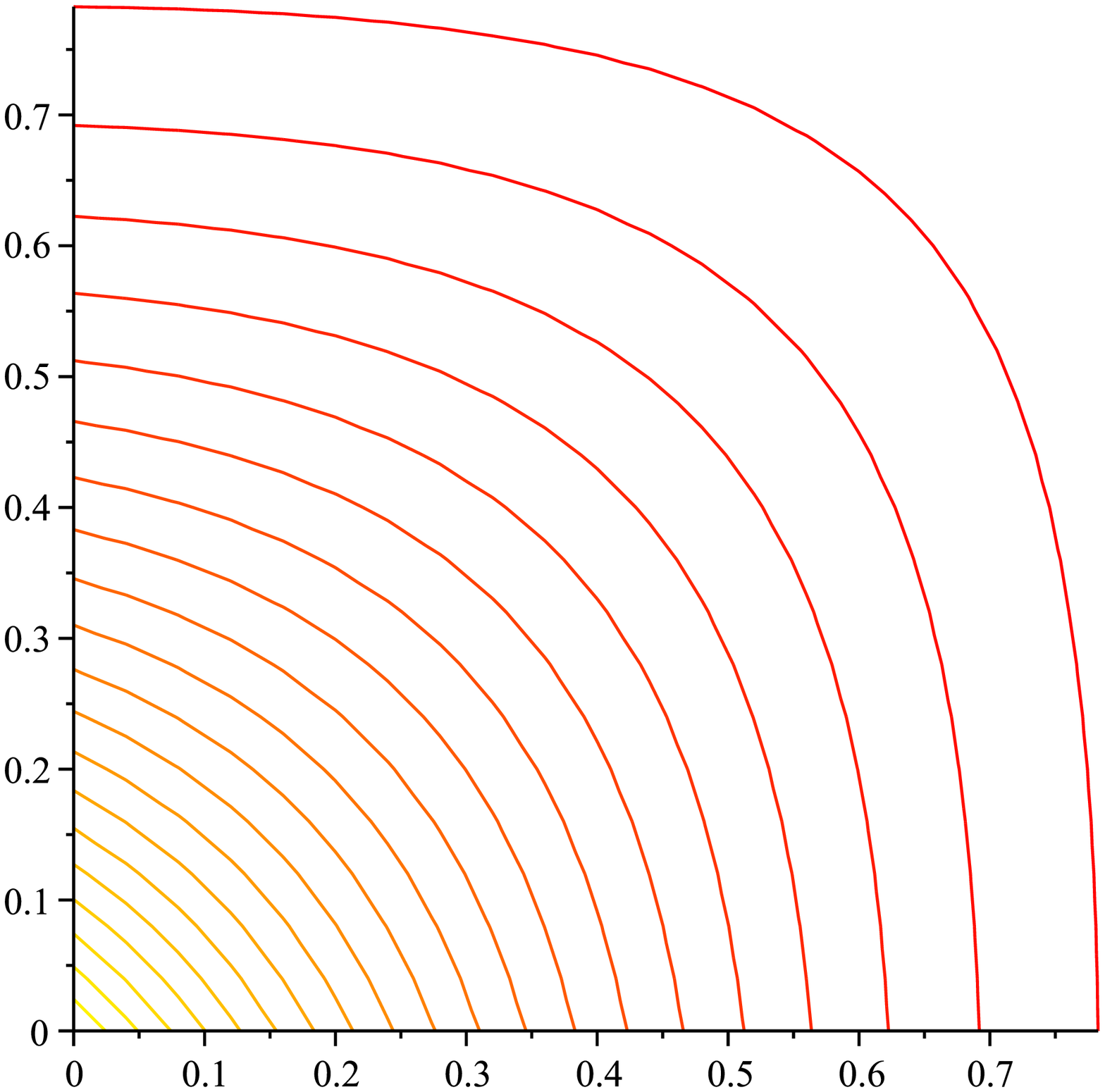}
			\caption{$\varrho_{ij}^2=0.5$}
			\label{fig_gamma:subfig4}
		\end{subfigure}
	}
	\caption{Plots and contour curves for the dependence hazard rate derivative $\gamma_{0,\{i,j\}}$ of a bivariate correlated frailty model 
	  based on a bivariate chi-squared distribution over the range $[0,1]^2$, with different values of $\varrho_{ij}^2$.}
\label{fig:gamma_0_fuer_Zj2_2}
\end{figure}
\section{References}
\bibliographystyle{elsarticle-harv}
\bibliography{references_Bendel_Dobler_Janssen_2014}

\end{document}